\documentclass[10pt]{article}
\pdfoutput=1
\usepackage[utf8]{inputenc}
\usepackage[english]{babel}
\usepackage{amsmath,amssymb,amsfonts,amsthm,amscd} 
\usepackage{mathtools}
\usepackage[T1]{fontenc}                                                      
\usepackage{graphicx}
\usepackage{lmodern}%
\usepackage{stmaryrd} 
\usepackage{color}
\usepackage{verbatim}
\usepackage{enumerate}
\usepackage{xifthen}
\usepackage{calc}
\usepackage[breaklinks=true,hidelinks]{hyperref}
\usepackage{hhline}
\usepackage{colortbl}
\usepackage{extarrows}
\usepackage[ruled]{algorithm2e} 
\usepackage{blindtext}
\usepackage{mathrsfs} 
\usepackage{import} 
\usepackage{mathtools}
\usepackage{tikz}
\usepackage{pgfplots}
\pgfplotsset{compat=1.14}
\usetikzlibrary{spy,calc,arrows.meta,calc,decorations.markings,math,arrows.meta,shapes.misc,fillbetween}
\tikzset{cross/.style={cross out, draw=black, fill=none, minimum size=2*(#1-\pgflinewidth), inner sep=0pt, outer sep=0pt}, cross/.default={2pt}}

\usepackage{fullpage}

\newtheorem{defi}{Definition}[section]
\newtheorem{theorem}[defi]{Theorem}
\newtheorem{lem}[defi]{Lemma}
\newtheorem{cor}[defi]{Corollary}

\makeatletter  
\def\@endtheorem{\qed\endtrivlist\@endpefalse } 
\makeatother

\newtheoremstyle{remarks}
{10pt} 
{10pt} 
{} 
{} 
{\bfseries} 
{.}
{ } 
{}

\theoremstyle{remarks}
\newtheorem{rem}[defi]{Remark}

\newcommand{\N}{\ensuremath{\mathbb{N}}}
\newcommand{\Z}{\ensuremath{\mathbb{Z}}}

\newcommand{\R}{\ensuremath{\mathbb{R}}}
\newcommand{\C}{\ensuremath{\mathbb{C}}}

\newcommand{\norm}[2][]{\ensuremath{\left\|#2\right\|_{#1}}} 

\newcommand{\snorm}[2][]{\ensuremath{|#2|_{#1}}}
\newcommand{\snormlr}[2][]{\ensuremath{\left|#2\right|_{#1}}}
\newcommand{\scalarprod}[2]{\ensuremath{\left\langle{#1,#2}\right\rangle}} 

\renewcommand{\d}{\, d}

\newcommand{\uU}{\underline{U}}
\newcommand{\uW}{\underline{W}}
\newcommand{\ucurlN}{\underline{\curlN}}
\newcommand{\uX}{\underline{X}}

\newcommand{\curlX}{\ensuremath{\mathcal{X}}}
\newcommand{\curlM}{\ensuremath{\mathcal{M}}}
\newcommand{\curlP}{\ensuremath{\mathcal{P}}}
\newcommand{\curlL}{\ensuremath{\mathcal{L}}}

\newcommand{\curlY}{\ensuremath{\mathcal{Y}}}
\newcommand{\curlZ}{\ensuremath{\mathcal{Z}}}

\newcommand{\curlU}{\ensuremath{\mathcal{U}}}
\newcommand{\curlE}{\ensuremath{\mathcal{E}}}
\newcommand{\curlN}{\ensuremath{\mathcal{N}}}
\newcommand{\curlV}{\ensuremath{\mathcal{V}}}
\newcommand{\curlW}{\ensuremath{\mathcal{W}}}
\newcommand{\curlO}{\ensuremath{\mathcal{O}}}

\newcommand{\real}{\operatorname{Re}}
\newcommand{\adj}{\operatorname{adj}}

\makeatletter
\newcommand{\pushright}[1]{\ifmeasuring@#1\else\omit\hfill$\displaystyle#1$\fi\ignorespaces}
\newcommand{\pushleft}[1]{\ifmeasuring@#1\else\omit$\displaystyle#1$\hfill\fi\ignorespaces}
\makeatother

\title{Modulating traveling fronts for the Swift-Hohenberg equation in the case of an additional conservation law}
\author{Bastian Hilder\footnote{Institut f{\"u}r Analysis, Dynamik und Modellierung, Universit{\"a}t Stuttgart, Pfaffenwaldring 57, 70569 Stuttgart, Germany. E-mail address: \href{mailto:bastian.hilder@mathematik.uni-stuttgart.de}{bastian.hilder@mathematik.uni-stuttgart.de} \newline \copyright 2020. This manuscript version is made available under the CC-BY-NC-ND 4.0 license \href{http://creativecommons.org/licenses/by-nc-nd/4.0/}{\nolinkurl{http://creativecommons.org/licenses/by-nc-nd/4.0/}}}}
\date{\today}

\usepackage{tocloft}
\begin{document}
\maketitle

\begin{abstract}
We consider the one-dimensional Swift-Hohenberg equation coupled to a conservation law. 
As a parameter increases the system undergoes a Turing bifurcation. 
We study the dynamics near this bifurcation. 
First, we show that stationary, periodic solutions bifurcate from a homogeneous ground state. 
Second, we construct modulating traveling fronts which model an invasion of the unstable ground state by the periodic solutions. 
This provides a mechanism of pattern formation for the studied system. 
The existence proof uses center manifold theory for a reduction to a finite-dimensional problem. 
This is possible despite the presence of infinitely many imaginary eigenvalues for vanishing bifurcation parameter since the eigenvalues leave the imaginary axis with different velocities if the parameter increases. 
Furthermore, compared to non-conservative systems, we address new difficulties arising from an additional neutral mode at Fourier wave number $k=0$ by exploiting that the amplitude of the conserved variable is small compared to the other variables.
\end{abstract}



\section{Introduction}\label{sec:introduction}

The dynamics of pattern forming systems on unbounded domains has been an important part of research for the last decades since such phenomena can be observed in many real-world examples from biology, chemistry and fluid dynamics, see \cite{crossHohenberg93}.
Prototypical examples of such systems are the Taylor-Couette problem and the Rayleigh-B{\'e}nard problem.
Typically, pattern formation occurs as follows.
When an external control parameter, e.g.\  the temperature for the Reyleigh-B{\'e}nard problem, increases beyond a critical value the spatially homogeneous equilibrium gets unstable and spatially periodic solutions bifurcate.

We are interested in the mechanism which drives the pattern formation in parabolic evolution equations on the real line.
It turns out that in such systems patterns often arise in the wake of an invading heteroclinic front which connects the unstable ground state $u_0$ to the periodic pattern $u_\text{per}$.
To model this behavior, we construct \emph{modulating traveling fronts} in this paper.
These are solutions of the form 
\begin{align}
	u(t,x) = U(x-ct,x-\beta t),
	\label{eq:modfrontAnsatz}
\end{align}
where $U$ is periodic with respect to its second argument and satisfies
\begin{align}
	\lim_{\xi \rightarrow \infty} U(\xi,p) = u_0(p) \text{ and } \lim_{\xi \rightarrow -\infty} U(\xi,p) = u_\text{per}(p),
	\label{eq:modfrontBC}
\end{align}
see Figure \ref{fig:modfront}.
Here, $x \in \R$ denotes the unbounded spatial direction of the problem, $c$ the velocity of the front, $t \geq 0$ the time and $\beta$ the phase velocity.
Solutions of this type are already established for non-conservative systems which exhibit a Turing bifurcation.
These results include the case of cubic nonlinearities such as the Swift-Hohenberg equation \cite{colletEckmann86,eckmannWayne91} as well as quadratic nonlinearities such as the Taylor-Couette problem in an infinite cylinder \cite{haragusSchneider99} and a nonlocal Fisher-KPP equation \cite{fayeHolzer15}.

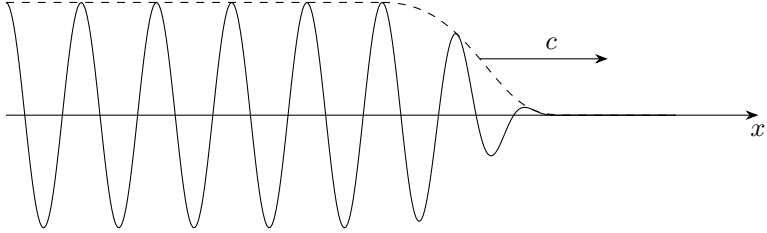
\begin{figure}
	\centering
	\begin{tikzpicture}
		\draw[-{Stealth[scale=1]}] (-5,0) -- (5,0) node[below] {$x$};
		\draw[domain=-5:0,samples=200] plot (\x, {1.5*cos(2*pi*\x r)});
		\draw[domain=-0:3.9,samples=200] plot (\x, {1.5*e^(-3*\x^2/(4-\x)^2)*cos(2*pi*\x r)});
		\draw[domain=-5:0,dashed] plot (\x, {1.5});
		\draw[domain=0:3.9,dashed,samples=100] plot (\x, {1.5*e^(-3*\x^2/(4-\x)^2)});
		\draw[-Stealth] (1.3,0.75) -- (2.25,0.75) node[above] {$c$} -- (3,0.75);
	\end{tikzpicture}
	\caption{Modulating travelling front}
	\label{fig:modfront}
\end{figure}

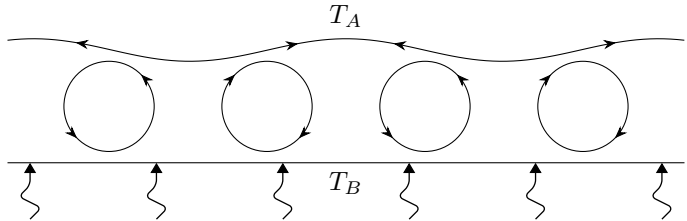
\begin{figure}
	\centering
	\begin{tikzpicture}[scale=1.5]
		\draw (-3,0.5) -- (3,0.5) node[below,pos=.5]{$T_B$};
		\draw[domain=-3:3,samples=100,
		decoration={markings,
			mark=at position 0.1 with {\arrowreversed[black]{Stealth}},
			mark=at position 0.43 with {\arrow[black]{Stealth}},
			mark=at position 0.57 with {\arrowreversed[black]{Stealth}},
			mark=at position 0.9 with {\arrow[black]{Stealth}}},
		postaction=decorate] plot(\x,{0.1*cos(130*\x)+1.5});
		\draw[decoration={markings, mark=at position 0.125 with {\arrow{Stealth}}, mark=at position 0.625 with {\arrow{Stealth}}},postaction={decorate}]	(-2.1,1) circle (0.4);
		\draw[decoration={markings, mark=at position 0.375 with {\arrowreversed{Stealth}}, mark=at position 0.875 with {\arrowreversed{Stealth}}},postaction={decorate}]	(-0.7,1) circle (0.4);
		\draw[decoration={markings, mark=at position 0.125 with {\arrow{Stealth}}, mark=at position 0.625 with {\arrow{Stealth}}},postaction={decorate}]	(0.7,1) circle (0.4);
		\draw[decoration={markings, mark=at position 0.375 with {\arrowreversed{Stealth}}, mark=at position 0.875 with {\arrowreversed{Stealth}}},postaction={decorate}]	(2.1,1) circle (0.4);
		\node[draw=none,fill=none] at (0,1.8) {$T_A$};
		\foreach \x in {-2.8000 ,  -1.6800 ,  -0.5600 ,   0.5600 ,   1.6800 ,   2.8000} {
			\draw[-Triangle,line width=0.5pt,rounded corners=2pt](\x,0)
			--++(0.1,0.1)
			--++(-0.2,0.1)
			--++(0.1,0.1)
			--++(0,0.2);
		}
	\end{tikzpicture}
	\caption{Schematic depiction of the B{\'e}nard-Marangoni convection close to the first instability, with $T_B > T_A$.}
	\label{fig:BenardMarangoni}
\end{figure}

In this paper, we show the existence of modulating traveling fronts for a pattern forming system with a conserved quantity which presents a nontrivial extension of the current results for non-conservative systems as we outline below.
Examples for pattern forming systems with a conserved quantity are the Bérnard-Marangoni convection (see Figure \ref{fig:BenardMarangoni}) and the flow down an inclined plane, see \cite{haeckerSchneiderZimmermann11}.
It turns out that the behavior of these systems differs from the one of classical pattern forming systems (see e.g.\ \cite{knobloch16}).
In particular, the Ginzburg-Landau equation cannot be justified as an amplitude equation as opposed to the classical situation where this is generically the case (see \cite{schneiderUecker17} for details).
Instead, a modified Ginzburg-Landau system, that is, a Ginzburg-Landau equation coupled to a conservation law, see \eqref{eq:amplitudeEquation1}--\eqref{eq:amplitudeEquation2}, appears as amplitude equation, see \cite{matthewsCox00}.
A closer analysis reveals that the additional conservation law in the amplitude equation comes from the presence of an additional critical spectral curve of the linearisation about the ground state touching the imaginary axis at Fourier wave number $k = 0$, see Figure \ref{fig:spectralSituation}.
This presents new difficulties in the rigorous justification of the amplitude equation which is done in \cite{haeckerSchneiderZimmermann11, schneiderZimmermann13, duellKashaniSchneiderZimmermann16} but also the dynamics of the modified Ginzburg-Landau system is less well understood as opposed to the classical Ginzburg-Landau equation (see \cite{schneiderZimmermann17} for a discussion).
Indeed, we will see that both issues lead to new difficulties for the problem at hand.

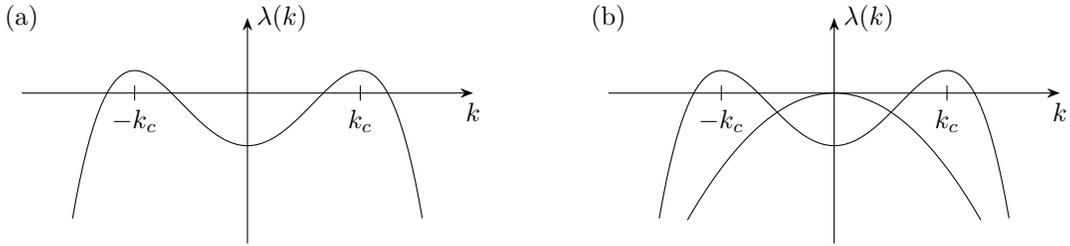
\begin{figure}
\centering
\begin{tikzpicture}[scale=1]
	\draw[-Stealth] (-3,0) -- (3,0) node[below] {$k$};
	\draw[-Stealth] (0,-2) -- (0,1) node[right] {$\lambda(k)$};
	\draw[-] (1.5,0.1)--(1.5,-0.1) node[below] {$k_c$};
	\draw[-] (-1.5,0.1)--(-1.5,-0.1) node[below] {$-k_c$};
	\draw[domain=-2.325:2.325,samples=100] plot(\x,{-(1-(\x*\x)/(1.5*1.5))*(1-(\x*\x)/(1.5*1.5))+0.3});
	\node at (-3,1) {(a)};
\end{tikzpicture}
\hspace{1cm}
\begin{tikzpicture}[scale=1]
	\draw[-Stealth] (-3,0) -- (3,0) node[below] {$k$};
	\draw[-Stealth] (0,-2) -- (0,1) node[right] {$\lambda(k)$};
	\draw[-] (1.5,0.1)--(1.5,-0.1) node[below] {$k_c$};
	\draw[-] (-1.5,0.1)--(-1.5,-0.1) node[below] {$-k_c$};
	\draw[domain=-2.325:2.325,samples=100] plot(\x,{-(1-(\x*\x)/(1.5*1.5))*(1-(\x*\x)/(1.5*1.5))+0.3});
	\draw[domain=-1.95:1.95,samples=100] plot(\x,{-(\x*\x)/(1.5*1.5)});
	\node at (-3,1) {(b)};
\end{tikzpicture}
\caption{Curves of eigenvalues over the Fourier wave number $k$ for a classical Turing instability (a) and the new type of instability with marginally stable modes at $k = 0$ (b).}
\label{fig:spectralSituation}
\end{figure}

\subsection{The problem and main result}

We consider the Swift-Hohenberg equation with an additional conservation law of the form
\begin{subequations}
	\begin{align}
		\partial_t u &= -(1+\partial_x^2)^2 u + \alpha u + uv - u^3,
		\label{eq:SHe}\\
		\partial_t v &= \partial_x^2 v + \gamma \partial_x^2(u^2),
		\label{eq:ConE}
	\end{align}
\end{subequations}
with $\alpha, \gamma, u(t,x), v(t,x) \in \R$ and $t \geq 0$ (see \cite{coxMatthews03,schneiderZimmermann17}).
Here $\alpha$ takes the role of an external control parameter.
The system \eqref{eq:SHe}--\eqref{eq:ConE} has a family of spatially homogeneous equilibria $(0, v_0)$ with $v_0 \in \R$ which destabilizes when $\alpha > v_0$.
However, we restrict the following analysis to $v_0 = 0$, since the general case can be brought back to this case by replacing $\alpha$ by $\alpha + v_0$.
Linearisation of \eqref{eq:SHe}--\eqref{eq:ConE} about $(0,0)$ and the Fourier transform yield that the spectral situation is indeed as depicted in Figure \ref{fig:spectralSituation}b.
Hence, at $\alpha = 0$ there is a Turing bifurcation and we will prove that a family of periodic solutions bifurcates.
Close to the first instability, namely for $\alpha = \varepsilon^2 \alpha_0$, these solutions are of the form
\begin{subequations}
	\begin{align}
		u_\text{per}(x) &= \varepsilon A_u \cos(x+x_0) + \curlO(\varepsilon^2) \label{eq:perSolUIntroduction}\\
		v_\text{per}(x) &= \varepsilon^2 A_v \cos(2(x+x_0)) + \curlO(\varepsilon^3), \label{eq:perSolVIntroduction}
	\end{align}
\end{subequations}
with $x_0 \in [0,2\pi)$ and $A_u,A_v \in \R$ (see Lemma \ref{lem:periodicSols}).

Note that although \eqref{eq:SHe}--\eqref{eq:ConE} is a purely phenomenological model, as discussed in \cite{schneiderZimmermann17}, it shares important properties with the Bérnard-Marangoni problem.
In particular, both models show the same kind of Turing instability and both are reflection symmetric.
Moreover, the (formal) amplitude equation of \eqref{eq:SHe}--\eqref{eq:ConE} is of the same form as the one for the Bénard-Marangoni problem \cite[Theorem 4.3.3]{zimmermann14} (see Section \ref{sec:periodicSolutions}).
Thus, we expect that both models exhibit similar behavior close to the first instability.

As mentioned before, we aim to rigorously construct a modulating traveling front for the above system.
In particular, we are interested in slow moving fronts with velocity $c = \varepsilon c_0$, which is the natural scaling corresponding to the scaling of $\alpha$.
Furthermore, since the system \eqref{eq:SHe}--\eqref{eq:ConE} possesses reflection symmetry and the periodic solutions \eqref{eq:perSolUIntroduction}--\eqref{eq:perSolVIntroduction} are stationary, we set the phase velocity $\beta = 0$.
The idea of the existence proof is to use center manifold theory to reduce the problem to finding a heteroclinic orbit in a finite-dimensional system. We now briefly outline the strategy, describe the challenges therein and formulate our main result.
Finally, we comment on the differences to previous results and novelties in our work.

Center manifold theory is a well-established tool to reduce the dimension of a studied system.
We refer to \cite{haragusIooss11} for a detailed overview.
In the case at hand, it can be applied in the following, non-standard way.
Inserting the modulating traveling front ansatz into \eqref{eq:SHe}--\eqref{eq:ConE} we find that the linearisation about the ground state has infinitely many imaginary eigenvalues for $\varepsilon = 0$.
Therefore, standard center manifold results are not applicable, since these require that only finitely many eigenvalues lie on the imaginary axis.
However, as $\varepsilon$ gets positive we find similar to \cite{eckmannWayne91} that the eigenvalues depart from the imaginary axis with different velocities and thus, a spectral gap of size $\curlO(\sqrt{\varepsilon})$ between finitely many eigenvalues close to the imaginary axis and the rest of the spectrum arises.
We use this to construct a 6-dimensional center manifold although, since the size of the spectral gap depends on $\varepsilon$, it is not a priori clear whether small (with respect to $\varepsilon$) solutions are contained in the center manifold.
Therefore, we prove that the constructed center manifold is large enough to contain the perodic solutions \eqref{eq:perSolUIntroduction}--\eqref{eq:perSolVIntroduction}.
Finally, we establish the existence of heteroclinic connections for the finite-dimensional system on the center manifold and arrive at the main result.

\begin{theorem}\label{thm:ModFrontsIntroduction}
	For $c_0^2 > 16\alpha_0 > 0$ there exist $\varepsilon_0 > 0$ and $\gamma_0 > 0$ such that for all $\varepsilon \in (0,\varepsilon_0)$ and $\snorm{\gamma} < \gamma_0$ the system \eqref{eq:SHe}--\eqref{eq:ConE} has modulating traveling front solutions of the form \eqref{eq:modfrontAnsatz} satisfying the boundary conditions \eqref{eq:modfrontBC}.
	Furthermore, the solution is of the form
	\begin{align*}
		u_f(t,x) &= \varepsilon U(\varepsilon(x-\varepsilon c_0 t)) \cos(x+x_0) + \curlO(\varepsilon^2) \\
		v_f(t,x) &= \varepsilon^2 \left[V_0(\varepsilon(x-\varepsilon c_0 t)) + V_1(\varepsilon(x-\varepsilon c_0 t)) \cos(2(x+x_0))\right] + \curlO(\varepsilon^3)
	\end{align*}
	for any $x_0 \in [0,2\pi)$, with $U,V_0,V_1 \in \R$.
\end{theorem}

\begin{rem}\label{rem:dependenceOnGamma}
	We point out that our rigorous result, Theorem \ref{thm:ModFrontsIntroduction}, only covers the case that $\gamma$ is contained in a small neighbourhood of zero.
	However, numerical experiments presented in Section \ref{sec:ReducedEquationsAndHeteroclinicConnections} show that we can expect that the result also holds for $\gamma \in (-3,\infty)$.
	Notably, we cannot expect that either small periodic patterns or small modulating fronts exist for $\gamma \leq -3$ since the reduced system does not have nontrivial fixed points in this parameter region.
	We refer to Section \ref{sec:periodicSolutions} for a detailed discussion.
\end{rem}

\begin{rem}\label{rem:differentModels}
	We note that a similar instability occurs in the models studied in \cite{matthewsCox00,schneiderZimmermann13,sukhtayev16ArXiv}.
	In fact, with a modulating front ansatz \eqref{eq:modfrontAnsatz} we find a similar spectral situation as for \eqref{eq:SHe}--\eqref{eq:ConE}, see Section 3.1.
	Namely, for $\varepsilon = 0$ infinitely many eigenvalues lie on the imaginary axis and depart with different velocities for $\varepsilon > 0$.
	Thus, we expect that similar results hold true in these cases.
\end{rem}

Although the strategy used to establish Theorem \ref{thm:ModFrontsIntroduction} is similar to the one used in \cite{colletEckmann86,eckmannWayne91,haragusSchneider99} the following new challenges occur.
First, the proof of the center manifold result, Theorem \ref{thm:centerManifold}, is more involved because of the additional spectral curve due to the conservation law, see Figure \ref{fig:spectralSituation}b.
Second, as can be expected from the corresponding amplitude equation, the reduced system on the center manifold is 3-dimensional instead of 2-dimensional which increases the difficulty of establishing a heteroclinic orbit.

We now comment on the differences and new ideas in the center manifold result in more detail.
Compared to the pure Swift-Hohenberg case \cite{colletEckmann86,eckmannWayne91}, new difficulties arise from the presence of quadratic nonlinearities instead of cubic ones.
Therefore, as discussed in \cite{haragusSchneider99}, the estimates for the semigroups and projections established in \cite{eckmannWayne91} seem insufficient to establish a large enough center manifold.
Haragus and Schneider \cite{haragusSchneider99} address this issue by rescaling the central and hyperbolic part with different powers of $\varepsilon$, i.e.\ $u = \varepsilon^\beta u_c + \varepsilon^\gamma u_h$, with $0 < \beta < 1 < \gamma < 2$ chosen appropriately, which is motivated by the fact that the corresponding periodic equilibrium has the same expansion.
Finally, they remove the critical quadratic terms using a normal form transformation.

However, it turns out that this is not a viable strategy in our case since not all critical quadratic terms can be eliminated using normal form transformation, see Remark \ref{rem:failiureNormalFormTransform}.
A closer analysis reveals that the resonating quadratic contributions come from the fact that quadratic interactions of central modes, i.e.\ modes with non-negative growth rate, give again central modes.
This is due to the additional spectral curve touching zero at $k = 0$ by virtue of the conserved quantity present in the system.
Note that this is also one of the main issues in the justification of the modified Ginzburg-Landau approximation of equations with conserved quantities in \cite{schneiderZimmermann13} since the quadratic interactions of central modes are not exponentially damped in contrast to pattern forming systems without conserved quantity.

In this paper, we solve the issues related to the quadratic contributions by using a similar rescaling as used by Haragus and Schneider \cite{haragusSchneider99}.
However, we additionally exploit $u_\text{per} = \curlO(\varepsilon)$ and $v_\text{per} = \curlO(\varepsilon^2)$ and thus we rescale $u$ and $v$ differently in terms of $\varepsilon$.
Moreover, we use the transformation $w = v + \gamma u^2$ which gives an additional $\varepsilon$ in the quadratic nonlinearity of the conservation law.
Note that this transformation does not eliminate the quadratic nonlinearites and hence is not a normal form transform.
It turns out that this is enough to treat all nonlinearities similar to cubic ones and therefore, normal form transform is not necessary, see also Remark \ref{rem:whydoesthiswork}.

\subsection{Outline}
The paper is organized as follows.
In Section \ref{sec:periodicSolutions} we formally derive the amplitude equation of \eqref{eq:SHe}--\eqref{eq:ConE} using a multiple scaling ansatz and show the existence of periodic solutions.
Following, we formulate the modulating front problem and prove a center manifold result including bounds on the size of the manifold in Section \ref{sec:problemAndCMR}.
Afterwards, Section \ref{sec:ReducedEquationsAndHeteroclinicConnections} contains the derivation of the reduced equations on the center manifold and the construction of heteroclinic orbits establishing the existence of modulating fronts for \eqref{eq:SHe}--\eqref{eq:ConE}.
Finally, in Section \ref{sec:discussion} we discuss related, open problems.

\section{Amplitude equation and spatially periodic equilibria}\label{sec:periodicSolutions}

We want to derive an amplitude equation for \eqref{eq:SHe}--\eqref{eq:ConE}.
For this, we note that the linearisation of \eqref{eq:SHe}--\eqref{eq:ConE} about the trivial solution $(u,v) = (0,0)$ has solutions of the form
\begin{align*}
	u(x,t) = e^{ikx} e^{\lambda_u(k) t} \text{ and } v(x,t) = e^{ikx} e^{\lambda_v(k) t},
\end{align*}
where $\lambda_u(k) = -(1-k^2)^2 + \alpha$ and $\lambda_v(k) = -k^2$, with $k \in \R$, which corresponds to the spectrum depicted in Figure \ref{fig:spectralSituation}b.
Therefore for $\alpha \leq 0$ the ground state $(u,v) = (0,0)$ is stable while for $\alpha > 0$ it gets unstable.
Since we are interested in solutions bifurcating from the ground state when $\alpha$ gets positive, we set $\alpha = \varepsilon^2 \alpha_0$ for some $\alpha_0 > 0$.
We now derive the formal amplitude equation close to the instability with the ansatz
\begin{align*}
	u(x,t) &= \varepsilon A(X, T) e^{ix} + \varepsilon \overline{A}(X, T) e^{-ix}, \\
	v(x,t) &= \varepsilon^2 B_0(X, T) + \varepsilon^2 B_1(X,T) e^{2ix} + \varepsilon^2 \overline{B}_1(X,T) e^{-2ix},
\end{align*}
where $X = \varepsilon x$ and $T = \varepsilon^2 t$.
Inserting this ansatz into the system \eqref{eq:SHe}--\eqref{eq:ConE} then yields
\begin{subequations}
    \begin{align}
        0 &= \varepsilon^3 E (-\partial_T A + 4 \partial_X^2 A + \alpha_0 A + A B_0 + \bar{A} B_1 - 3 \snorm{A}^2 A) + \varepsilon^3 E^3 (AB_1 + A^3) + c.c. + \curlO(\varepsilon^4), \label{eq:formalDerivation1}\\
        0 &= \varepsilon^4 (-\partial_T B_0 + \partial_X^2 B_0 + 2 \gamma \partial_X^2(\snorm{A}^2)) + \varepsilon^4 E^2(-\partial_T B_1 + \partial_X^2 B_1 + \gamma \partial_X^2 (A^2)) \nonumber\\
        &\qquad+ 4 i \varepsilon^3 E^2 (\partial_X B_1 + \gamma \partial_X (A^2)) - 4 \varepsilon^2 E^2 (B_1 + \gamma A^2) + c.c. + \curlO(\varepsilon^5),\label{eq:formalDerivation2}
    \end{align}
\end{subequations}
where $E = e^{ix}$.
Equating the $\varepsilon^2 E^2$ contributions in \eqref{eq:formalDerivation2} to zero then yields $B_1 = -\gamma A^2$.
With this, equating the $\varepsilon^3 E$-terms in \eqref{eq:formalDerivation1} and the $\varepsilon^4$-terms in \eqref{eq:formalDerivation2} to zero leads to the formal amplitude equations
\begin{subequations}
	\begin{align}
		\partial_T A &= 4 \partial_X^2 A + \alpha_0 A + A B_0 - (3+\gamma) A \snorm{A}^2\label{eq:amplitudeEquation1}, \\
		\partial_T B_0 &= \partial_X^2 B_0 + 2\gamma \partial_X^2 \snorm{A}^2. \label{eq:amplitudeEquation2}
	\end{align}
\end{subequations}
Note that the extended ansatz for $v$ includes noncritical modes, i.e., modes which correspond to negative eigenvalues.
However, these additional terms are necessary in the derivation to equate the coefficient in front of $\varepsilon^2 e^{2ix}$ in \eqref{eq:formalDerivation2} to zero.
This contribution then leads to a nontrivial contribution to the cubic term in \eqref{eq:amplitudeEquation1}.
Note that this system deviates from the one derived in \cite{matthewsCox00} since the sign of the cubic coefficient in \eqref{eq:amplitudeEquation1} depends on the coupling parameter $\gamma$.
Specifically, for $\gamma < -3$ the sign is positive which leads to the non-existence of nontrivial equilibria in \eqref{eq:amplitudeEquation1}--\eqref{eq:amplitudeEquation2}.
Indeed this formal prediction reflects in the full system as for $\gamma \leq -3$ the full system \eqref{eq:SHe}--\eqref{eq:ConE} does not have stationary periodic solutions of small amplitude.
Making these formal ideas rigorous we obtain the following result.

\begin{lem}\label{lem:periodicSols}
	For all $\gamma_0 > -3$ and $\alpha_0 > 0$ there exists an $\varepsilon_0 > 0$ such that for all $\varepsilon \in (0,\varepsilon_0)$ and $\gamma \in (\gamma_0,\infty)$ the system \eqref{eq:SHe}--\eqref{eq:ConE} has stationary periodic solutions of the form
	\begin{subequations}
		\begin{align}
			u_\text{per}(x) &= \varepsilon 2 \sqrt{\dfrac{\alpha_0 - q_0^2}{3 + \gamma}} \cos(k_c (x + x_0)) + \mathcal{O}(\varepsilon^2), \label{eq:Uper} \\
			v_\text{per}(x) &= -\varepsilon^2 2 \dfrac{(\alpha_0 - q_0^2) \gamma}{3+\gamma} \cos( 2k_c(x+x_0)) + \mathcal{O}(\varepsilon^3), \label{eq:Vper}
		\end{align}
	\end{subequations}
	where $x_0 \in [0,2\pi/k_c)$ and $k_c^2 = 1 + \varepsilon q_0$.
	Here, $k_c$ is the critical wave number and $q_0 \in (-\sqrt{\alpha_0}, \sqrt{\alpha_0})$.
	Furthermore, $\int_0^{2\pi/k_c} v_\text{per} \d x = 0$.
\end{lem}
\begin{proof}
    The proof is based on center manifold reduction.
    We define the space
    \begin{align*}
        H^l_\text{per} := \left\{u(x) = \sum_{n \in \Z} u_n e^{ink_c x} \,:\, \norm[H^l_\text{per}]{u}^2 = \sum_{n \in \Z} (1+n^2)^l \snorm{u_n}^2 < \infty\right\} 
    \end{align*}
    and
	\begin{align*}
		\curlZ &= \left\{(u,v) \in H^4_\text{per} \times H^2_\text{per} \,:\, u_n = \bar{u}_{-n},\, v_n = \bar{v}_{-n},\, v_0 = 0\right\}, \\
		\curlY &= \left\{(u,v) \in H^2_\text{per} \times H^2_\text{per} \,:\, u_n = \bar{u}_{-n},\, v_n = \bar{v}_{-n},\, v_0 = 0\right\}, \\
		\curlX &= \left\{(u,v) \in H^0_\text{per} \times H^0_\text{per} \,:\, u_n = \bar{u}_{-n},\, v_n = \bar{v}_{-n},\, v_0 = 0\right\},
	\end{align*}
	where $u_n$ denotes the $n$-th Fourier coefficient of $u$, $n \in \Z$ and $\bar{u}$ is the complex conjugate of $u$.
	Furthermore, note that the condition $v_0 = \int_0^{2\pi/k_c} v(x) \d x = 0$ is conserved by the dynamic of \eqref{eq:SHe}--\eqref{eq:ConE}.
	Then, we write \eqref{eq:SHe}--\eqref{eq:ConE} as
	\begin{align*}
		\partial_t \begin{pmatrix}
			u \\ v
		\end{pmatrix} = L_\varepsilon \begin{pmatrix}
			u \\ v
		\end{pmatrix} + R(u,v),
	\end{align*}
	where $L_\varepsilon$ is the linear part and $R$ the nonlinear part.
	Here, the linear operator $L_\varepsilon$ is given by
	\begin{align*}
		L_\varepsilon = \begin{pmatrix}
			-(1+\partial_x^2)^2 + \varepsilon^2 \alpha_0 & 0 \\
			0 & \partial_x^2
		\end{pmatrix} \in \curlL(\curlZ; \curlX),
	\end{align*}
	where $\curlL(\curlZ; \curlX)$ is the space of bounded, linear operators from $\curlZ$ to $\curlX$.
	Furthermore, the nonlinearity is given by
	\begin{align*}
		R(u,v) = \begin{pmatrix}
			uv - u^3 \\
			\gamma \partial_x^2 (u^2)
		\end{pmatrix},
	\end{align*}
	which is a smooth map from $\curlZ$ to $\curlY$.
	Note in particular that $R(0,0) = 0$ and $DR(0,0) = 0$.
	The eigenvalues of $L_\varepsilon$ can be explicitly calculated using Fourier transform as
	\begin{align*}
		\lambda_{u,n} &= -(1-n^2 k_c^2)^2 + \varepsilon^2 \alpha_0, \quad n \in \Z \\
		\lambda_{v,n} &= -n^2 k_c^2, \quad n \in \Z \setminus \{0\}.
	\end{align*}
	Note, that $\lambda_{v,0} = 0$ is not an eigenvalue since we require $v_0 = 0$.
	Thus, the spectrum can be decomposed in the central part $\sigma_c(L_\varepsilon) = \{\lambda_{u,\pm 1}\}$ and the remaining hyperbolic part $\sigma_h(L_\varepsilon)$.
	Finally, we have for $\snorm{\omega} \geq \omega_0 > 0$ that
	\begin{align*}
		\norm[\curlL(\curlX;\curlX)]{(i\omega I - L_\varepsilon)^{-1}} \leq C(\sup_{n \in \Z} \snorm{i\omega + (1-n^2k_c^2)^2 - \varepsilon^2 \alpha_0}^{-1} + \sup_{n \in \Z}\snorm{i\omega - n^2 k_c^2}^{-1}) \leq \dfrac{C(\omega_0)}{\snorm{\omega}}.
	\end{align*}
	Thus, since $\curlZ$, $\curlY$ and $\curlX$ are Hilbert spaces, we can apply the center manifold theorem \cite[Theorem 3.3]{haragusIooss11} using \cite[Remark 3.6]{haragusIooss11}.
	This gives a smooth map $h = h(u_1,u_{-1},\varepsilon) = h_u \times h_v$, which defines the center manifold.
	We then introduce the following coordinates
	\begin{align*}
		u(t) &= \varepsilon A(\varepsilon^2 t) e^{ik_c x} + \varepsilon \bar{A}(\varepsilon^2 t) e^{-ik_c x} + h_u(\varepsilon A(\varepsilon^2 t), \varepsilon \bar{A}(\varepsilon^2 t),\varepsilon), \\
		v(t) &= h_v(\varepsilon A(\varepsilon^2 t), \varepsilon \bar{A}(\varepsilon^2 t),\varepsilon).
	\end{align*}
	To determine an approximation of $h_u$, $h_v$ we use \cite[Corollary 2.12]{haragusIooss11} and obtain
	\begin{align*}
		\varepsilon^2 D_A h_u \partial_T A &= -(1+\partial_x^2)^2 h_u + \varepsilon^2 \alpha_0 h_u + P_{h,u}\left((\varepsilon A e^{ik_c x} + \varepsilon \bar{A} e^{-i k_c x} + h_u)h_v + (\varepsilon A e^{ik_c x} + \varepsilon \bar{A} e^{-i k_c x} + h_u)^3\right),
	\end{align*}
	where $P_{h,u}$ is the projection onto the hyperbolic eigenspace of the $u$-equation, $D_Ah_u$ denotes the derivative of $h_u$ with respect to $A$ and $T = \varepsilon^2 t$.
	Using that $h_u, h_v = \curlO(\varepsilon^2)$ since they are at least quadratic in $A$, the above equation yields that $h_u = \curlO(\varepsilon^3)$.
	For $h_v$ we find
	\begin{align*}
		\varepsilon^3 D_A h_v \partial_T A = \partial_x^2 h_v + \partial_x^2((\varepsilon A e^{ik_c x} + \varepsilon \bar{A} e^{-i k_c x} + h_u)^2),
	\end{align*}
	where we used that all $\lambda_{v,n}$ are in the hyperbolic part of the spectrum.
	Since $h_u = \curlO(\varepsilon^3)$ we find by equating the $\varepsilon^2$-terms in the above equation to zero that
	\begin{align*}
		h_v = -\gamma \varepsilon^2(A^2 e^{2ik_c x} + \bar{A}^2 e^{-2ik_c x}) + \curlO(\varepsilon^4).
	\end{align*}
	Thus, the reduced equation on the center manifold for $A$ is given by
	\begin{align}
		\partial_T A = (\alpha_0 - q_0^2) A - (3+\gamma) A\snorm{A}^2 + g(A, \varepsilon),
		\label{eq:redDyn}
	\end{align}
	with $g(A,\varepsilon) = \curlO(\varepsilon^2)$.
	For $\varepsilon = 0$ and $\gamma > -3$, \eqref{eq:redDyn} has a nontrivial stationary state $\snorm{A}^2 = \frac{\alpha_0 - q_0^2}{3+\gamma}$.
	It remains to show that there exists a stationary state nearby for $\varepsilon > 0$ small.
	We proceed similar to \cite[Theorem 13.2.2]{schneiderUecker17} and use the translational invariance of \eqref{eq:SHe}--\eqref{eq:ConE}, which yields that \eqref{eq:redDyn} is invariant with respect to $A \mapsto A e^{iy}$, $y \in \R$.
	Therefore, $g(A,\varepsilon) = A \tilde{g}(\snorm{A}^2,\varepsilon)$ and hence, we may write $A$ in polar coordinates, i.e. $A = r e^{i\phi}$, and obtain
	\begin{align*}
		\partial_T r = (\alpha_0 - q_0^2)r - (3+\gamma) r^3 + r \tilde{g}(r^2,\varepsilon), \quad \partial_T \phi = 0.
	\end{align*}
	Thus, any stationary state satisfies
	\begin{align*}
		G(r,\varepsilon) := (\alpha_0 - q_0^2) r - (3+\gamma) r^3 + r\tilde{g}(r^2,\varepsilon) = 0.
	\end{align*}
	Since
	\begin{align*}
		G\left(\pm \sqrt{(\alpha_0-q_0^2)/(3+\gamma)},0\right) = 0 \text{ and } \partial_r G\left(\pm \sqrt{(\alpha_0-q_0^2)/(3+\gamma)},0\right) \neq 0
	\end{align*}
	we can apply the implicit function theorem to obtain stationary states $r_\pm(\varepsilon) = \pm \sqrt{(\alpha_0-q_0^2)/(3+\gamma)} + \curlO(\varepsilon^2)$.
	Using the coordinates on the center manifold, this concludes the proof.
\end{proof}

\section{Formulation of the problem and center manifold reduction}\label{sec:problemAndCMR}
We now turn to the existence of modulating front solutions of the form \eqref{eq:modfrontAnsatz} connecting the trivial solution $(u,v) = (0,0)$ to the periodic solution established in Lemma \ref{lem:periodicSols}.
The aim of this section is to establish the spatial dynamics formulation and show that there is a center manifold which contains the periodic solutions.
Thus, we make the ansatz
\begin{align*}
	u(t,x) = \curlU(\xi,p) \text{ and } v(t,x) = \curlV(\xi,p),
\end{align*}
where $\xi = x - ct$ and $p = x$, and $\curlU, \curlV$ are $2\pi / k_c$-periodic with respect to their second argument and satisfy
\begin{align*}
	\lim_{\xi \rightarrow +\infty} (\curlU, \curlV)(\xi,p) = (0,0) \text{ and } \lim_{\xi \rightarrow -\infty} (\curlU,\curlV)(\xi,p) = (u_\text{per}, v_\text{per})(p).
\end{align*}
Inserting this ansatz into \eqref{eq:SHe}--\eqref{eq:ConE} and using $\partial_t = -c \partial_\xi$ and $\partial_x = \partial_\xi + \partial_p$ we obtain
\begin{align*}
	-c \partial_\xi \curlU &= -(1+(\partial_\xi +\partial_p)^2)^2 \curlU + \alpha \curlU + \curlU \curlV - \curlU^3, \\
	-c \partial_\xi \curlV &= (\partial_\xi + \partial_p)^2 \curlV + \gamma (\partial_\xi + \partial_p)^2 (\curlU^2).
\end{align*}
Note that the nonlinearity contains $p$-derivatives.
To avoid this, we introduce the transformation $\curlW := \curlV + \gamma \curlU^2$ which gives the equivalent transformed system
\begin{subequations}
	\begin{align}
		-c \partial_\xi \curlU &= -(1+(\partial_\xi +\partial_p)^2)^2 \curlU + \alpha \curlU + \curlU \curlW - (1 + \gamma) \curlU^3, \label{eq:SHeTransformed}\\
		-c \partial_\xi \curlW &= (\partial_\xi + \partial_p)^2 \curlW - \gamma c \partial_\xi (\curlU^2). \label{eq:ConTransformed}
	\end{align}
\end{subequations}
Another consequence of the transformation $\curlW := \curlV + \gamma \curlU^2$ is the presence of the front velocity $c$ in the nonlinear term in \eqref{eq:ConTransformed}.
Especially with the choice $c = \varepsilon c_0$ this depends on $\varepsilon$.
It turns out this is crucial for the construction of a center manifold of sufficient size, see Remark \ref{rem:whydoesthiswork}.

\subsection{Spatial dynamics formulation and spectrum}\label{subsec:spatialformulation}

Since $\curlU(\xi,\cdot), \curlW(\xi,\cdot)$ are $2\pi/k_c$-periodic, we make a Fourier ansatz
\begin{align*}
	\curlU(\xi, p) = \sum_{n \in \Z} \curlU_n(\xi) e^{ink_cp} \text{ and } \curlW(\xi, p) = \sum_{n \in \Z} \curlW_n(\xi) e^{ink_cp},
\end{align*}
where $\curlU_n = \bar{\curlU}_{-n}$ and $\curlW_n = \bar{\curlW}_{-n}$ since $\curlU, \curlW$ are real.
Inserting this ansatz into \eqref{eq:SHeTransformed}--\eqref{eq:ConTransformed} gives an infinite dimensional ODE system with respect to the dynamic variable $\xi$
\begin{subequations}
	\begin{align}
		-c \partial_\xi \curlU_n &= -(1+(\partial_\xi +ink_c)^2)^2 \curlU_n + \alpha \curlU_n + \sum_{l+k = n} \curlU_l \curlW_k - (1 + \gamma) \sum_{k+l+m=n} \curlU_k \curlU_l \curlU_m, \label{eq:SHeODE}\\
		-c \partial_\xi \curlW_n &= (\partial_\xi + ink_c)^2 \curlW_n - \gamma c \partial_\xi \sum_{k+l=n} \curlU_k \curlU_l, \label{eq:ConODE}
	\end{align}
\end{subequations}
with $n \in \Z$. 
We write \eqref{eq:SHeODE} and \eqref{eq:ConODE} as a first order system
\begin{align}
	\partial_\xi \begin{pmatrix}
		U_n \\ W_n
	\end{pmatrix} = L_n \begin{pmatrix}
		U_n \\ W_n
	\end{pmatrix} + \curlN_n(U,W), 
	\label{eq:spatDynSys}
\end{align}
where $(U_n,W_n)^T = (U_{n0},U_{n1},U_{n2},U_{n3},W_{n0},W_{n1})^T \in \C^6$ with $U_{nj} = \partial_\xi^j \curlU_n$ for $j = 0,1,2,3$ and $W_{nj} = \partial_\xi^j \curlW_n$ for $j = 0,1$.
Since all coupling terms are nonlinear, the linear part $L_n$ is given by
\begin{align*}
	L_n = \begin{pmatrix}
	L_n^\text{SH} & 0 \\
	0 & L_n^\text{con}
	\end{pmatrix} = \begin{pmatrix}
		0 & 1 & 0 & 0 & 0 & 0 \\
		0 & 0 & 1 & 0 & 0 & 0 \\
		0 & 0 & 0 & 1 & 0 & 0 \\
		A_n & B_n & C_n & D_n & 0 & 0 \\
		0 & 0 & 0 & 0 & 0 & 1 \\
		0 & 0 & 0 & 0 & G_n & H_n
	\end{pmatrix} \in \C^{6 \times 6},
\end{align*}
where $A_n = -(1-\mu^2)^2 + \alpha$, $B_n = -4i\mu (1-\mu^2) + c$, $C_n = 6\mu^2 - 2$, $D_n = -4i\mu$, $G_n = \mu^2$ and $H_n = 2i\mu - c$.
Here, we used the abbreviation $\mu = nk_c$ for notational simplicity.
The nonlinearity $\curlN_n$ is given by
\begin{align*}
	\curlN_n(U,W) =	\left(0,0,0,\sum_{p+q=n} U_{p0} W_{q0} - (1+\gamma) \sum_{p+q+r=n} U_{p0} U_{q0} U_{r0}, 0, 2 c \gamma \sum_{p+q =n} U_{p0} U_{q1}\right)^T
\end{align*}
for all $n \in \Z$.

\subsubsection{Spectrum of the $L_n$}
Following \cite{eckmannWayne91}, we set $k_c = 1$ for simplicity in what follows.
This yields $\mu = n$.
However, the results remain true for $k_c = 1 + \varepsilon^2 q_0$, with $q_0 \in (-\sqrt{\alpha_0}, \sqrt{\alpha_0})$ similar to Lemma \ref{lem:periodicSols}.
Recall that the $L_n$ have a blockmatrix structure, i.e.\ 
\begin{align*}
	L_n = \begin{pmatrix}
		L_n^\text{SH} & 0 \\
		0 & L_n^\text{con}
	\end{pmatrix},
\end{align*}
with $L_n^\text{SH}$ originates from \eqref{eq:SHeODE} and $L_n^\text{con}$ originates from \eqref{eq:ConODE}.
Thus, we can analyse the spectrum for each block individually.
To do so, we set $\alpha = \varepsilon^2 \alpha_0$ and $c = \varepsilon c_0$ for some $\alpha_0, c_0 > 0$.
We find for $\varepsilon = 0$ that $L_n^\text{SH}$ has two double eigenvalues $\lambda_{n,\pm} = -i(n \pm 1)$ corresponding to a Jordan block of size 2.
For $\varepsilon > 0$, we find with $\Delta = \sqrt{c_0^2 - 16\alpha_0}$ that
\begin{align}
	\lambda_{n,+}^\pm &= \begin{dcases}
			\varepsilon \dfrac{-c_0 \pm \Delta}{8} + \curlO(\varepsilon^2), & n = -1, \\
			-i(n + 1) \pm \varepsilon^{1/2} i^{3/2} \dfrac{\sqrt{c_0(n + 1)}}{2} + \curlO(\varepsilon), & n \neq -1,
		\end{dcases} \label{eq:eigenvalueSH1}\\
	\lambda_{n,-}^\pm &= \begin{dcases}
			\varepsilon \dfrac{-c_0 \pm \Delta}{8} + \curlO(\varepsilon^2), & n = 1, \\
			-i(n - 1) \pm \varepsilon^{1/2} i^{3/2} \dfrac{\sqrt{c_0(n - 1)}}{2} + \curlO(\varepsilon), & n \neq 1.
		\end{dcases}
		\label{eq:eigenvalueSH2}
\end{align}
Similarly, we find for $L_n^\text{con}$ that for $n \neq 0$ the eigenvalues are given by
\begin{align}
	\nu_{n,\pm} = in \pm \varepsilon^\frac{1}{2} i^\frac{3}{2} \sqrt{n c_0} + \curlO(\varepsilon)
	\label{eq:eigenvalueCon}
\end{align}
while for $n = 0$ we have two eigenvalues $\nu_{0,+} = 0$ and $\nu_{0,-} = -\varepsilon c_0$.
Here, the zero eigenvalue at $n = 0$ comes from the fact that \eqref{eq:ConODE} is a conservation law.
The calculation leading to \eqref{eq:eigenvalueSH1}--\eqref{eq:eigenvalueCon} are given in Appendix \ref{app:eigenvalues}.

Summarizing, there are 6 ``central'' eigenvalues with real part $\curlO(\varepsilon)$ while the rest of the spectrum has real part $\curlO(\varepsilon^{1/2})$.
This is illustrated in Figure \ref{fig:spectrum}.
We thus have a similar mechanism as for the pure Swift-Hohenberg equation.
For $\varepsilon = 0$ we have a purely imaginary spectrum and as $\varepsilon$ gets positive the eigenvalues depart from the imaginary axis with different velocities which allows us to construct a center manifold.
However, the size of this center manifold will depend on $\varepsilon$ and the remainder of this section is devoted to the construction of a center manifold that is large enough to contain the spatially periodic equilibria from Lemma \ref{lem:periodicSols}.

\begin{figure}
	\centering
	\begin{tikzpicture}[scale=1]
		\fill[domain=2:3,samples=100,color={rgb:black,1;white,2},fill={rgb:black,1;white,2}] (2,0) -- plot(\x,{(\x-1.75)*(\x-1.75)+0.9375}) -- (3,0) -- cycle;
		\fill[domain=2:3,samples=100,color={rgb:black,1;white,2},fill={rgb:black,1;white,2}] (2,0) -- plot(\x,{-(\x-1.75)*(\x-1.75)-0.9375}) -- (3,0) -- cycle;
		\fill[domain=-3:-2,samples=100,color={rgb:black,1;white,2},fill={rgb:black,1;white,2}] (-3,0) -- plot(\x,{(\x+1.75)*(\x+1.75)+0.9375}) -- (-2,0) -- cycle;
		\fill[domain=-3:-2,samples=100,color={rgb:black,1;white,2},fill={rgb:black,1;white,2}] (-3,0) -- plot(\x,{-(\x+1.75)*(\x+1.75)-0.9375}) -- (-2,0) -- cycle;
		\draw[-Stealth] (-3,0) -- (3,0) node[below] {$\operatorname{Re}$};
		\draw[-Stealth] (0,-3) -- (0,3) node[right] {$\operatorname{Im}$};
		\draw[domain=2:3,samples=100] plot(\x,{(\x-1.75)*(\x-1.75)+0.9375});
		\draw[domain=2:3,samples=100] plot(\x,{-(\x-1.75)*(\x-1.75)-0.9375});
		\draw[-] (2,-1) -- (2,1);
		\draw[domain=-3:-2,samples=100] plot(\x,{(\x+1.75)*(\x+1.75)+0.9375});
		\draw[domain=-3:-2,samples=100] plot(\x,{-(\x+1.75)*(\x+1.75)-0.9375});
		\draw[-] (-2,-1) -- (-2,1);
		
		\draw[Stealth-Stealth,dashed] (0,1)--(1,1) node[above]{$\mathcal{O}(\varepsilon^{1/2})$}--(2,1);
		
		\draw (-0.5,0) node[cross=3]{};
		\draw (-0.5,0) circle (3pt);
		\draw (-1,0) node[cross=3]{};
		\draw (-1,0) circle (3pt);
		\draw[Stealth-Stealth,dashed] (0,0.5)--(-0.5,0.5) node[above]{$\mathcal{O}(\varepsilon)$}--(-1,0.5);
	\end{tikzpicture}
	\hspace{1cm}
	\begin{tikzpicture}[scale=1]
		\fill[domain=2:3,samples=100,color={rgb:black,1;white,2},fill={rgb:black,1;white,2}] (2,0) -- plot(\x,{(\x-1.75)*(\x-1.75)+0.9375}) -- (3,0) -- cycle;
		\fill[domain=2:3,samples=100,color={rgb:black,1;white,2},fill={rgb:black,1;white,2}] (2,0) -- plot(\x,{-(\x-1.75)*(\x-1.75)-0.9375}) -- (3,0) -- cycle;
		\fill[domain=-3:-2,samples=100,color={rgb:black,1;white,2},fill={rgb:black,1;white,2}] (-3,0) -- plot(\x,{(\x+1.75)*(\x+1.75)+0.9375}) -- (-2,0) -- cycle;
		\fill[domain=-3:-2,samples=100,color={rgb:black,1;white,2},fill={rgb:black,1;white,2}] (-3,0) -- plot(\x,{-(\x+1.75)*(\x+1.75)-0.9375}) -- (-2,0) -- cycle;
		\draw[-Stealth] (-3,0) -- (3,0) node[below] {$\operatorname{Re}$};
		\draw[-Stealth] (0,-3) -- (0,3) node[right] {$\operatorname{Im}$};
		\draw[domain=2:3,samples=100] plot(\x,{(\x-1.75)*(\x-1.75)+0.9375});
		\draw[domain=2:3,samples=100] plot(\x,{-(\x-1.75)*(\x-1.75)-0.9375});
		\draw[-] (2,-1) -- (2,1);
		\draw[domain=-3:-2,samples=100] plot(\x,{(\x+1.75)*(\x+1.75)+0.9375});
		\draw[domain=-3:-2,samples=100] plot(\x,{-(\x+1.75)*(\x+1.75)-0.9375});
		\draw[-] (-2,-1) -- (-2,1);
		\draw[Stealth-Stealth,dashed] (0,1)--(1,1) node[above]{$\mathcal{O}(\varepsilon^{1/2})$}--(2,1);
		
		\draw (-0.75,0) node[cross=3]{};
		\draw (0,0) node[cross=3]{};
		\draw[Stealth-Stealth,dashed] (0,0.5)--(-0.375,0.5) node[above]{$\mathcal{O}(\varepsilon)$}--(-0.75,0.5);
	\end{tikzpicture}
	\caption{Left: spectrum of $L^\text{SH}$ with double eigenvalues, Right: spectrum of $L^\text{con}$.}
	\label{fig:spectrum}
\end{figure}
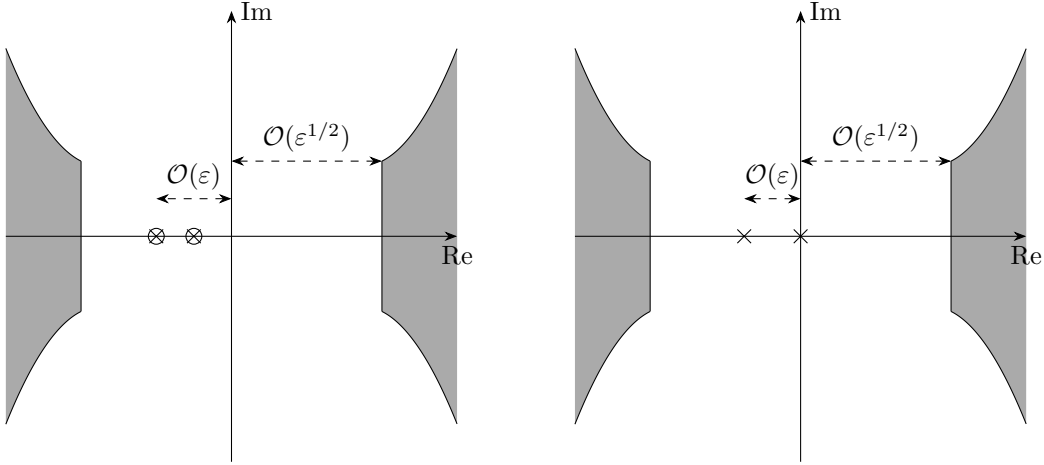

\subsubsection{The function space and spectrum of $L$}

We follow \cite{eckmannWayne91} and define the function space
\begin{align*}
	H^l(\C^m) := \left\{U(p) = \sum_{n \in \Z} U_{n} e^{ink_cp} \in \C^m: \norm[H^l]{U} < \infty\right\},
\end{align*}
where $l > 0$, $k_c$ the critical wave number as in Lemma \ref{lem:periodicSols} and $\norm[H^l]{\cdot}$ is defined by
\begin{align*}
	\norm[H^l]{U}^2 := \sum_{n \in \Z} (1+n^2)^l \snorm{U_n}^2.
\end{align*}
This space is a Hilbert space for $l > 0$ and a Banach algebra for $l > 1/2$.
To simplify the notation, we define $\curlE_{l_1,l_2} := H^{l_1}(\C^4) \times H^{l_2}(\C^2)$ which is equipped with the norm $\norm[\curlE_{l_1,l_2}]{(U,W)} = \norm[H^{l_1}]{U} + \norm[H^{l_2}]{W}$ for $U \in H^{l_1}(\C^4)$ and $W \in H^{l_2}(\C^2)$.
Additionally, we introduce extra notation for the case $l_1=l_2$, i.e.\ we denote $\curlE_l := \curlE_{l,l}$.
Using this notation, we can write \eqref{eq:SHeTransformed}--\eqref{eq:ConTransformed} equivalently as
\begin{align}
	\partial_\xi \begin{pmatrix}
		U \\ W
	\end{pmatrix} = L \begin{pmatrix}
	U \\ W
	\end{pmatrix} + \curlN(U,W),
	\label{eq:fullSystem}
\end{align}
where $(U,W) \in \curlE_l$, and the linear operator $L : \curlE_{l+4,l+2} \rightarrow \curlE_l$ and the nonlinearity $\curlN : \curlE_l \rightarrow \curlE_l$ are defined as
\begin{align*}
    L \begin{pmatrix}
	U \\ W
	\end{pmatrix} = \sum_{n \in \Z} L_n \begin{pmatrix}
	U_n \\ W_n
	\end{pmatrix} e^{ink_cp} \text{ and } \curlN(U,W) = \sum_{n \in \Z} \curlN_n(U,W) e^{ink_c p}.
\end{align*}
Here $L_n$ and $\curlN_n$ are given in \eqref{eq:spatDynSys}.
In particular we note that to calculate the spectrum of $L$ it is sufficient to calculate the spectra of $L_n$ for $n \in \Z$ since $\sigma(L) = \bigcup_{n \in \Z} \sigma(L_n)$.
Furthermore, since $\curlE_l$ is a Banach algebra for $l > 1/2$ and the nonlinearity $\curlN$ has a polynomial structure we have the following result.

\begin{cor}\label{cor:Lipschitz}
	Let $r > 0$ and $l > 1/2$ be arbitrary and $\gamma \in \R$ the coupling parameter in \eqref{eq:SHe}--\eqref{eq:ConE}.
	Then $\curlN : \curlE_l \rightarrow \curlE_l$ is Lipschitz continuous on $B_r(0)$ with 
	\begin{align*}
		\norm[\curlE_l]{\curlN(U_1,W_1) - \curlN(U_2,W_2)} \leq C (1+\snorm{\gamma}) r \norm[\curlE_l]{(U_1,W_1) - (U_2,W_2)}
	\end{align*}
	for some constant $C < \infty$ independent of $\gamma$ and $r$.
\end{cor}

\subsection{Center manifold theorem}
Due to the spectral situation, cf. Figure \ref{fig:spectrum}, the size of the center manifold depends on $\varepsilon$.
Therefore, in order to estimate its size, we have to recall the proof of the center manifold theorem.
As mentioned in the introduction, our proof is inspired by the one in \cite{haragusSchneider99}.
First, we introduce the projections on the $\curlO(\varepsilon)$-center part of the spectrum using the Dunford integral
\begin{align*}
	P_{n,c} = \dfrac{1}{2\pi i} \int_{\Gamma_{n,c}} (\lambda I - L_n)^{-1} \d \lambda,
\end{align*}
where $\Gamma_{n,c}$ is a smooth, simple curve surrounding the central spectrum of $L_n$ which is well defined since $L_n \in \C^{6\times 6}$.
Especially, we have $P_{n,c} = 0$ for $n \notin \{0,\pm 1\}$.
Furthermore, we define the projection on the $\curlO(\varepsilon^{1/2})$-hyperbolic part of the spectrum $P_{n,h} = I - P_{n,c}$.
Using that the central and the hyperbolic spectra are $\curlO(1)$-separated -- they have a different imaginary part of distance $\curlO(1)$ -- these operators can be bounded independently of $\varepsilon$ and $n$.
We define for $(U,W) \in \curlE_l$
\begin{align}
	\curlP_c \begin{pmatrix}
		U \\ W
	\end{pmatrix} = \sum_{n \in \Z} P_{n,c} \begin{pmatrix}
		U_n \\ W_n
	\end{pmatrix} e^{ink_c p},
	\label{eq:centralProjection}
\end{align}
which satisfies the following corollary.

\begin{cor}
	Let $l > 1/2$.
	There exists an $\varepsilon_0 > 0$ and $C < \infty$ independent of $\varepsilon_0$ such that
	\begin{align*}
		\norm[\curlL(\curlE_l)]{\curlP_c} < C
	\end{align*}
	for all $\varepsilon \in (0,\varepsilon_0)$, where $\curlL(\curlE_l)$ denotes the space of linear, bounded maps from $\curlE_l$ into itself.
\end{cor}

With this, we introduce the central part of the operator $L_c := \curlP_c L$ and $L_{n,c} := P_{n,c} L_n$ and we define $L_h$ and $L_{n,h}$ analogously.
Furthermore, we introduce $(U_c,W_c) = \curlP_c (U,W)$ and $(U_h,W_h) = \curlP_h (U,W)$.
We recall that the periodic solutions are of different order in terms of $\varepsilon$, namely $u_\text{per} = \curlO(\varepsilon)$ and $w_\text{per} = v_\text{per} + \gamma u_\text{per}^2 = \curlO(\varepsilon^2)$.
Keeping this in mind, we introduce the rescaling
\begin{align}
	U_c = \varepsilon^{\gamma_u} \underline{U}_c, \ U_h = \varepsilon^{\beta_u} \underline{U}_h, \ W_c = \varepsilon^{\gamma_w} \underline{W}_c, \ W_h = \varepsilon^{\beta_w} \underline{W}_h,
	\label{eq:rescaling}
\end{align}
with $\gamma_u < 1$, $\beta_u, \gamma_w < 2$ and $\beta_w < 3$.
Hence, by projecting \eqref{eq:fullSystem} onto the center and hyperbolic eigenspaces and inserting the rescaling \eqref{eq:rescaling} we obtain
\begin{subequations}
	\begin{align}
		\partial_\xi \begin{pmatrix}
		\uU_c \\ \uW_c
		\end{pmatrix} &= L_c \begin{pmatrix}
		\uU_c \\ \uW_c
		\end{pmatrix} + \ucurlN^\text{quad}_c(\uU_c, \uU_h, \uW_c, \uW_h) +  \ucurlN^\text{cub}_c(\uU_c, \uU_h, \uW_c, \uW_h), \label{eq:SHrescaled}\\
		\partial_\xi \begin{pmatrix}
		\uU_h \\ \uW_h
		\end{pmatrix} &= L_h \begin{pmatrix}
		\uU_h \\ \uW_h
		\end{pmatrix} + \ucurlN^\text{quad}_h(\uU_c, \uU_h, \uW_c, \uW_h) +  \ucurlN^\text{cub}_h(\uU_c, \uU_h, \uW_c, \uW_h),
		\label{eq:Conrescaled}
	\end{align}
\end{subequations}
with the rescaled nonlinearities given by
\begin{align*}
	\ucurlN^\text{quad}_c &= \curlP_c \begin{pmatrix}
		0 \\
		0 \\
		0 \\
		\varepsilon^{-\gamma_u}(e^{\gamma_u+\gamma_w} (\uU_c)_0 (\uW_c)_0 + \varepsilon^{\gamma_u + \beta_w} (\uU_c)_0 (\uW_h)_0 + \varepsilon^{\beta_u+\gamma_w} (\uU_h)_0 (\uW_c)_0 + \varepsilon^{\beta_u+\beta_w} (\uU_h)_0 (\uW_h)_0) \\
		0 \\
		\varepsilon^{1-\gamma_w} 2\gamma c_0 (e^{2\gamma_u} (\uU_c)_0 (\uU_c)_1 + \varepsilon^{\gamma_u + \beta_u} ((\uU_c)_0 (\uU_h)_1 + (\uU_h)_0 (\uU_c)_1) + \varepsilon^{2\beta_u} (\uU_h)_0 (\uU_h)_1)
	\end{pmatrix}, \\
	\ucurlN^\text{quad}_h &= \curlP_h \begin{pmatrix}
	0 \\
	0 \\
	0 \\
	\varepsilon^{-\beta_u}(e^{\gamma_u+\gamma_w} (\uU_c)_0 (\uW_c)_0 + \varepsilon^{\gamma_u + \beta_w} (\uU_c)_0 (\uW_h)_0 + \varepsilon^{\beta_u+\gamma_w} (\uU_h)_0 (\uW_c)_0 + \varepsilon^{\beta_u+\beta_w} (\uU_h)_0 (\uW_h)_0) \\
	0 \\
	\varepsilon^{1-\beta_w} 2\gamma c_0 (e^{2\gamma_u} (\uU_c)_0 (\uU_c)_1 + \varepsilon^{\gamma_u + \beta_u} ((\uU_c)_0 (\uU_h)_1 + (\uU_h)_0 (\uU_c)_1) + \varepsilon^{2\beta_u} (\uU_h)_0 (\uU_h)_1)
	\end{pmatrix}, \\
	\ucurlN^\text{cub}_c &= \curlP_c(0,0,0,\varepsilon^{-\gamma_u}(\varepsilon^{\gamma_u} (\uU_c)_0 + \varepsilon^{\beta_u} (\uU_h)_0)^3,0,0)^T, \\
	\ucurlN^\text{cub}_h &= \curlP_h(0,0,0,\varepsilon^{-\beta_u}(\varepsilon^{\gamma_u} (\uU_c)_0 + \varepsilon^{\beta_u} (\uU_h)_0)^3,0,0)^T. 
\end{align*}
Here, we used that $c = \varepsilon c_0$.

\begin{rem}\label{rem:failiureNormalFormTransform}
	Note that in contrast to \cite{haragusSchneider99}, the quadratic interaction of the central modes does not vanish in $\curlN_c^\text{quad}$.
	Furthermore, the novel term cannot be eliminated via normal form transform with an $\curlO(1)$-bound for $\varepsilon \rightarrow 0$.
	To see this we recall the nonresonance condition
	\begin{align*}
	\snorm{\lambda_1 + \lambda_2 - \lambda_3} > C
	\end{align*}
	uniformly for all small $\varepsilon > 0$ for some $C$ independent of $\varepsilon$, see \cite{haragusSchneider99} and the references therein.
	Here $\lambda_3$ corresponds to the eigenvalue of the equation in which the quadratic term should be eliminated and $\lambda_1$, $\lambda_2$ correspond to the eigenvalues belonging to the variables in the quadratic term.
	However, since we want to eliminate quadratic terms of the form $(\underline{U}_c)_0 (\underline{U}_c)_1$ in the equation for $\underline{W}_c$ and $(\underline{U}_c)_0(\underline{W}_c)_0$ in the equation for $\underline{U}_c$ we find that $\lambda_i \rightarrow 0$ for $\varepsilon \rightarrow 0$ and $i = 1,2,3$ since all central eigenvalues vanish for $\varepsilon = 0$.
	In particular, their imaginary part vanishes.
	Thus the nonresonance condition is violated and we cannot eliminate these terms with a bounded normal form transform.
\end{rem}

From Corollary \ref{cor:Lipschitz}, we then obtain the following result.

\begin{lem}\label{lem:LipschitzNonlin}
	Let $r > 0$, $l > 1/2$, $\varepsilon > 0$ and and $\gamma \in \R$ the coupling parameter in \eqref{eq:SHe}--\eqref{eq:ConE}.
	Then, $\ucurlN_c := \ucurlN^\text{quad}_c + \ucurlN^\text{cub}_c : \curlE_l \times \curlE_l \rightarrow \curlE_l$ and $\ucurlN_h := \ucurlN^\text{quad}_h + \ucurlN^\text{cub}_h : \curlE_l \times \curlE_l \rightarrow \curlE_l$ are Lipschitz continuous on $B_r(0) \subset \curlE_l$ with
	\begin{align*}
			\norm[\curlE_l \times \curlE_l]{\ucurlN_c(\uX_1) - \ucurlN_c(\uX_2)} &\leq C (1+\snorm{\gamma}) r \varepsilon^{\min(\gamma_w, \beta_w, \beta_u + \gamma_w - \gamma_u, 1+ 2\gamma_u-\gamma_w)} \norm[\curlE_l \times \curlE_l]{\uX_1 - \uX_2}, \\
			\norm[\curlE_l \times \curlE_l]{\ucurlN_h(\uX_1) - \ucurlN_h(\uX_2)} &\leq C (1+\snorm{\gamma}) r \varepsilon^{\min(\gamma_u+\gamma_w - \beta_u, \gamma_w,\beta_w,1+2\gamma_u-\beta_w)} \norm[\curlE_l \times \curlE_l]{\uX_1 - \uX_2},
	\end{align*}
	where $\uX = (\uU_c, \uU_h, \uW_c,\uW_h)$ and $C < \infty$ independent of $\varepsilon$ and $\gamma$.
\end{lem}

It turns out that this is enough to prove that the system \eqref{eq:SHrescaled}--\eqref{eq:Conrescaled} has a center manifold of size $\curlO(1)$ for a certain choice of $\gamma_u, \gamma_w, \beta_u, \beta_w$.
This then yields the desired result.
To show this, we define the semigroup corresponding to the stable, unstable and central part of $L_n$ by
\begin{align*}
	S_{n,c}(t) &= \dfrac{1}{2\pi i} \int_{\Gamma_{n,c}} (\lambda I - L_n)^{-1} e^{\lambda t} \d \lambda, \\
	S_{n,s}(t) &= \dfrac{1}{2\pi i} \int_{\Gamma_{n,s}} (\lambda I - L_n)^{-1} e^{\lambda t} \d \lambda, \\
	S_{n,u}(t) &= \dfrac{1}{2\pi i} \int_{\Gamma_{n,u}} (\lambda I - L_n)^{-1} e^{\lambda t} \d \lambda,
\end{align*}
where $\Gamma_{n,c}, \Gamma_{n,s}$ and $\Gamma_{n,u}$ are smooth, simple curves around the central, stable and unstable part of the spectrum of $L_n$, respectively.
And similarly to $\curlP_c$ we define
\begin{align*}
	S_{j}(t) X = \sum_{n \in \Z} S_{n,j}(t) X_n e^{ink_c p}
\end{align*}
for $X \in \curlE_l$ and $j \in \{c,s,u\}$.
We now provide estimates for these operators.

\begin{lem}\label{lem:lipschitzConti}
	There exist constants $C_1, C_2 < \infty$ and $\varepsilon_0 > 0$ such that for all $\varepsilon \in (0, \varepsilon_0)$ the following estimates hold
	\begin{align*}
		\norm[\curlL(\curlE_l)]{S_{c}(t)} &\leq C_1 (1 + \snorm{t}) e^{C_2 \varepsilon \snorm{t}}, \\
		\norm[\curlL(\curlE_l)]{S_{u}(t)} &\leq C_1 \dfrac{1}{\sqrt{\varepsilon}} e^{-C_2 \sqrt{\varepsilon} \snorm{t}} \text{ for } t < 0, \\
		\norm[\curlL(\curlE_l)]{S_{s}(t)} &\leq C_1 \dfrac{1}{\sqrt{\varepsilon}} e^{-C_2 \sqrt{\varepsilon} \snorm{t}} \text{ for } t > 0.
	\end{align*}
\end{lem}
\begin{proof}
	For the estimate of the central semigroup we remark that the central spectrum is $\curlO(1)$-bounded away from the rest of the spectrum for $\varepsilon$ sufficiently small.
	Hence, for any small $\varepsilon_0 > 0$ we can fix the curves $\Gamma_{n,c}$ for all $\varepsilon \in (0,\varepsilon_0)$.
	Now, we estimate the semigroup $S_{n,c}$ for $n \in \{\pm 1,0\}$. 
	This is sufficient since $\sigma(L_n)$ does not contain central eigenvalues for $n \notin \{0,\pm 1\}$.
	Since $L_n$ has a block structure, we can decompose the residual $(\lambda - L_n)^{-1}$ into $(\lambda - L_n^\text{SH})^{-1}$ and $(\lambda - L_n^\text{con})^{-1}$ and thus, we can estimate both parts separately.
	For $n = 0$, the Swift-Hohenberg part has no central eigenvalues and we focus on $L_0^\text{con} \in \C^{2\times 2}$ which has only central eigenvalues.
	We rewrite the residual as $(\lambda - L_0^\text{con})^{-1} = \det(\lambda - L_0^\text{con})^{-1} \adj(\lambda - L_0^\text{con})$, where $\adj(A)$ denotes the adjunct of $A$ and obtain by explicitly calculating the adjunct
	\begin{align*}
		\norm[\C^{2\times 2}]{\dfrac{1}{2\pi i} \int_{\Gamma_{0,c}} \dfrac{1}{\lambda (\lambda + \varepsilon c_0)} \adj(\lambda - L_0^\text{con}) e^{\lambda t} \d\lambda} &\leq \dfrac{1}{2\pi i} \snormlr{\int_{\Gamma_{0,c}} \left[\dfrac{e^{\lambda t}}{\lambda} + \dfrac{e^{\lambda t}}{\lambda + \varepsilon c_0} + \dfrac{e^{\lambda t}}{\lambda(\lambda + \varepsilon c_0)} \right]\d\lambda} \\
		&\leq C_1 (1 + \snorm{t}) e^{C_2 \varepsilon \snorm{t}}.
	\end{align*}
	By similar calculations, we obtain the same estimate for $n = \pm 1$ since $L_n^\text{SH}$ has two central and two hyperbolic eigenvalues.
	Thus, since the estimates are independent of $n \in \Z$, we obtain the upper bound
	\begin{align*}
		\norm[\curlE_l]{S_c(t)X} = \left(\sum_{n \in \Z} (1+n^2)^l \snorm{S_{n,c}(t) X}^2\right)^{1/2} \leq C (1 + \snorm{t}) e^{C_2 \varepsilon \snorm{t}} \norm[\curlE_l]{X}
	\end{align*}
	for any $X \in \curlE_l$ which yields the estimate for $S_c(t)$.
	
	Next, we prove the estimate for the stable and unstable semigroup, respectively.
	We focus on the stable semigroup since the proofs are identical.
	Note that we can decompose the spectrum of $L_n$ into three pairs of eigenvalues which are $\curlO(1)$ apart from each other, see \eqref{eq:eigenvalueSH1}--\eqref{eq:eigenvalueSH2} and \eqref{eq:eigenvalueCon}.
	Thus, for any small $\varepsilon > 0$ we can decompose the curve $\Gamma_{n,s}$ into three separate curves.
	However, these cannot be chosen independent of $\varepsilon$ since the eigenvalues are symmetric with respect to the imaginary axis and their real part vanishes for $\varepsilon \rightarrow 0$.
	We now estimate the contribution of $L_n^\text{con}$ and remark that the contribution of $L_n^\text{SH}$ can be handled analogously since there are only two relevant eigenvalues as mentioned above.
	Hence, for $n \neq 0$ we calculate for $t > 0$
	\begin{align*}
		\norm[\C^{2\times 2}]{\dfrac{1}{2\pi i} \int_{\Gamma_{n,s}} (\lambda I - L^\text{con}_n)^{-1} e^{\lambda t} \d \lambda} &= \norm[\C^{2\times 2}]{\dfrac{1}{2\pi i} \int_{\Gamma_{n,s}} \dfrac{1}{\lambda (\lambda + \epsilon c_0)} \adj(\lambda - L_0^\text{con}) e^{\lambda t} \d\lambda} \\
		&\leq \dfrac{1}{2\pi i} \norm[\C^{2\times 2}]{\int_{\Gamma_{n,s}} \left[\dfrac{e^{\lambda t}}{(\lambda - \nu_{n,+})(\lambda - \nu_{n,-})} \adj(\lambda - L_n^\text{SH})\right] \d\lambda} \\
		&\leq \norm[C^{2\times 2}]{\adj(\nu_{n,-} - L_n^\text{SH})} \dfrac{1}{\snorm{\nu_{n,-} - \nu_{n,+}}} e^{\nu_{n,-} t} \\
		&\leq C_1 \dfrac{1}{\sqrt{\varepsilon}} e^{-C_2 \sqrt{\varepsilon}t},
	\end{align*}
	where we used that $\real(\nu_{n,\pm}) = \curlO(\varepsilon^{1/2})$ and $\real(\nu_{n,-}) < 0$.
	Since this estimate is independent of $n \in \Z$ we obtain the desired estimate for $S_{s}$ in $\curlL(\curlE_l)$.
\end{proof}

We now introduce a smooth cut-off function $\chi_r \in [0,1]$ for some $r > 0$ which we define by $\chi_r(x) = 1$ for $\snorm{x} < r/2$, $\chi_r = 0$ for $\snorm{x} > r$ and $\chi_r \in (0,1)$ for $\snorm{x} \in [r/2,r]$.
Using this cut-off function, we define $\tilde{\ucurlN}_j := \chi_r \ucurlN_j$ for $j \in \{c,h\}$ which is a globally Lipschitz continuous mapping with the estimates given in Lemma \ref{lem:lipschitzConti}.
We now split $\curlP_h = \curlP_s + \curlP_u$ into a projection onto the stable eigenspaces $\curlP_s$ and a projection onto the unstable eigenspaces $\curlP_u$ which we define similarly to $\curlP_c$.
Then, we follow the standard procedure for the construction of center manifold results and show that the mapping 
\begin{align}
X_c(t) &= S_c(t) X_c(0) + \int_0^t S_c(t-\tau) \tilde{\ucurlN}_c(X(\tau)) \d \tau \nonumber\\
X_u(t) &= - \int_t^\infty S_u(t - \tau) \tilde{\ucurlN}_h(X(\tau)) \d\tau \label{eq:fixedPointSystem} \\
X_s(t) &= \int_{-\infty}^{t} S_s(t - \tau) \tilde{\ucurlN}_h(X(\tau)) \d \tau \nonumber
\end{align}
has a fixed point in
\begin{align*}
\curlX_\eta := \left\{X \in C^0(\R, \curlE_l) : \norm[\eta]{X} := \sup_{t \in \R} e^{-\eta \snorm{t}} \norm[\curlE_l]{X} < \infty \right\}
\end{align*}
with $\eta = \frac{C_2}{2} \sqrt{\varepsilon}$ where we use the decomposition $X_{j} = \curlP_j X$ for $j \in \{c,s,u\}$.
Using Lemma \ref{lem:LipschitzNonlin} we now estimate
\begin{subequations}
	\begin{align}
		\norm[\eta]{\tilde{\ucurlN}_c(X_1) - \tilde{\ucurlN}_c(X_2)} &\leq C (1+\snorm{\gamma}) r \varepsilon^{\min(\gamma_w, \beta_w, \beta_u + \gamma_w - \gamma_u, 1+ 2\gamma-\gamma_w)} \norm[\eta]{X_1 - X_2}, \label{eq:LipschitzEstimateC}\\
		\norm[\eta]{\tilde{\ucurlN}_h(X_1) - \tilde{\ucurlN}_h(X_2)} &\leq C (1+\snorm{\gamma}) r \varepsilon^{\min(\gamma_u+\gamma_w - \beta_u, \gamma_w,\beta_w,1+2\gamma_u-\beta_w)} \norm[\eta]{X_1 - X_2}, \label{eq:LipschitzEstimateH}
	\end{align}
\end{subequations}
for any $X_1,X_2 \in \curlX_\eta$.
Furthermore, for any nonlinear mapping $V : \curlE_l \rightarrow \R$ we have
\begin{align*}
	\norm[\eta]{\int_0^t S_c(t-\tau) V(X(\tau)) \d\tau} &\leq \sup_{t \in \R} e^{-\eta\snorm{t}} \int_0^t C_1 (1+\snorm{t}) e^{C_2 \varepsilon \snorm{t-\tau}} e^{\eta \tau} \d\tau \norm[\eta]{V(X)} \\
	&\leq C \varepsilon^{-1} \norm[\eta]{V(X)}, \\
	\norm[\eta]{\int_t^\infty S_u(t-\tau) V(X(\tau)) \d\tau} &\leq \sup_{t \in \R} e^{-\eta \snorm{t}} \int_t^\infty C_1 \dfrac{1}{\sqrt{\varepsilon}} e^{-C_2 \sqrt{\varepsilon}\snorm{t - \tau}} e^{\eta \snorm{\tau}} \d\tau \norm[\eta]{V(X)} \\
	&\leq C \varepsilon^{-1} \norm[\eta]{V(X)}, \\
	\norm[\eta]{\int_{-\infty}^t S_s(t-\tau) V(X(\tau)) \d\tau} &\leq C \varepsilon^{-1} \norm[\eta]{V(X)},
\end{align*}
where we used $\int_0^\infty e^{\sqrt{\varepsilon}t} t \d t = \curlO(\varepsilon^{-1})$ and $\eta < C_2 \sqrt{\varepsilon}$.
Combining these estimates we find that the system \eqref{eq:fixedPointSystem} maps $\curlX_\eta$ into itself and that we can estimate the Lipschitz constant of $\eqref{eq:fixedPointSystem}$ which is of order
\begin{align*}
	C(1+\snorm{\gamma}) r \varepsilon^{\min(\kappa_1,\kappa_2)-1},
\end{align*}
where $\kappa_1,\kappa_2$ are the exponents in \eqref{eq:LipschitzEstimateC} and \eqref{eq:LipschitzEstimateH}, respectively.
Finally, we set 
\begin{align}
	\gamma_u = 2/3 + \delta \text{, } \beta_u = 1 + \delta/2 \text{, } \gamma_w = 4/3 \text{ and } \beta_w = 4/3 + \delta
	\label{eq:rescalingCoeffs}
\end{align} 
for some $\delta > 0$.
This choice yields that $\kappa_1, \kappa_2 > 1$ and thus, \eqref{eq:fixedPointSystem} is a contraction in $\curlX_\eta$ provided that $\varepsilon$ is small enough.
Especially, this does not impose a restriction on the cut-off radius $r$.
Thus following the standard proof of the center manifold theorem, see, e.g.\ \cite{haragusIooss11}, we have a center manifold of size $\curlO(1)$ for the rescaled system \eqref{eq:SHrescaled}--\eqref{eq:Conrescaled}.
Finally, by reverting the rescaling \eqref{eq:rescaling} we arrive at the main result of this section.

\begin{theorem}\label{thm:centerManifold}
	Let $l > 1/2$, $\delta \in \left(0, 1/3\right)$ and $\gamma \in \R$.
	There exists an $\varepsilon_0 > 0$ such that for every $\varepsilon \in (0,\varepsilon_0)$ there exists a neighbourhood $O_c = O_u \times O_w \subset \curlE_c := \curlP_c \curlE_l$ of the origin and a mapping $h = (h_u, h_w) : O_c \rightarrow \curlE_h := \curlP_h\curlE_l$ such that the following holds.
	\begin{itemize}
		\item The neighbourhood $O_u$ is of size $\mathcal{O}(\varepsilon^{2/3 + \delta})$ and $O_w$ is of size $\mathcal{O}(\varepsilon^{4/3})$.
		\item The center manifold
		\begin{align*}
		\curlM_c = \{(U,W) = (U_c,W_c)+ h(U_c,W_c) : (U_c, W_c) \in O_c\}
		\end{align*}
		contains all small bounded solutions of \eqref{eq:fullSystem}.
		\item Every solution of the reduced system
		\begin{align}
			\partial_\xi \begin{pmatrix}
				U_c \\ W_c
			\end{pmatrix} &= L_c \begin{pmatrix}
		U_c \\ W_c
		\end{pmatrix} + \curlP_c \curlN(U_c + h_u(U_c,W_c), W_c + h_w(U_c,W_c))
		\label{eq:CMEquation}
		\end{align}
		gives a solution to the full system~\eqref{eq:fullSystem} via $(U,W) = (U_c,W_c) + h(U_c,W_c)$.
	\end{itemize}
\end{theorem}

\begin{rem}
		We point out that due to the ansatz \eqref{eq:rescaling}, the center manifold in Theorem \ref{thm:centerManifold} has different sizes in $u$ and $w$-direction, respectively.
		However, this is sufficient for the construction of modulating fronts.
		To see this, we recall that the periodic solution established in Lemma \ref{lem:periodicSols} satisfies $u_\text{per} = \curlO(\varepsilon)$ and $w_\text{per} := v_\text{per} + \gamma u_\text{per}^2 = \curlO(\varepsilon^2)$.
		We will see that the modulating front solutions have the same property and thus, the center manifold constructed above is sufficiently large to contain modulating fronts.
\end{rem}

\begin{rem}\label{rem:whydoesthiswork}
	As seen above, the construction of the center manifold does not require a normal form transform as in \cite{haragusSchneider99}.
	The reason for this lies both in the different rescaling of the $U$ and $W$ contribution in \eqref{eq:rescaling} and the transformation $w = v + \gamma u^2$ which gives an additional $\varepsilon$ in front of the quadratic term in \eqref{eq:ConTransformed}.
	Using this we see that $U_c W_c = \curlO(\varepsilon^{\gamma_u+\gamma_w})$ and $c \partial_\xi(U_c^2) = \curlO(\varepsilon^{1+2\gamma_u})$.
	Therefore, both critical quadratic nonlinearities scale better than $\varepsilon^2$ for $\gamma_u, \gamma_w$ chosen in \eqref{eq:rescalingCoeffs} and can be treated without further issues.
\end{rem}

\section{Reduced equations and heteroclinic connections}\label{sec:ReducedEquationsAndHeteroclinicConnections}

We now derive the reduced equations on the center manifold constructed in Theorem \ref{thm:centerManifold} and establish the existence of heteroclinic orbits connecting a circle of non-trivial fixed points to the origin.
The reduction mainly follows \cite{eckmannWayne91}, however, the construction of the heteroclinic orbits requires more work since we obtain an additional equation corresponding to the conservation law.
The main idea for this construction is to handle first the case $\gamma = 0$ which corresponds to the pure Swift-Hohenberg equation and afterwards use perturbation arguments to establish the persistence of these orbits for $\gamma \neq 0$ close to $0$.
Finally we numerically consider the case for large $\gamma$ in which we also find heteroclinic orbits.

\subsection{Derivation of the reduced equations}

We start by introducing the following coordinates for the central modes
\begin{align*}
	\begin{pmatrix}
		U_c(\xi,p) \\ W_c(\xi)
	\end{pmatrix} 
	&= \varepsilon \left(A(\varepsilon\xi) \varphi_1^+ + B(\varepsilon\xi) \varphi_1^-\right) e^{ik_c p} + \varepsilon \left(\overline{A}(\varepsilon\xi) \overline{\varphi_1^+} + \overline{B}(\varepsilon\xi) \overline{\varphi_1^-}\right) e^{-ik_c p} \\
	&\qquad + (0,0,0,0,\varepsilon^2 W_{c0}(\varepsilon\xi), \varepsilon^3 W_{c1}(\varepsilon\xi))^T, 
\end{align*}
where $A,B \in \C$, $W_{c0},W_{c1} \in \R$ and $\varphi_1^+,\varphi_1^- \in \C^6$ are the eigenvectors corresponding to the eigenvalues $\lambda_{1,-}^\pm$ of $L_{1}$ computed in \eqref{eq:eigenvalueSH2}.
Here, we normalize the eigenvectors $\varphi_1^{\pm}$ such that the first component is equal to one.
Note that since $L_1$ has a block diagonal structure, the last two components of $\varphi_1^{\pm}$ are zero.
Additionally, we used that $U_1 = \bar{U}_{-1}$ -- this gives the complex conjugated terms in $U_c$ -- and that all eigenvalues of $L_0^\text{con}$ are central.
Furthermore, we find $h_u \in \curlO(\varepsilon^3)$ and $h_w \in \curlO(\varepsilon^4)$.
To see this, we recall that
\begin{align*}
    Dh_u \partial_\xi X_c &= L_h^\text{SH} h_u + \curlP^\text{SH}_h(\curlN^\text{SH}(X_c + h(X_c))), \\
    Dh_w \partial_\xi X_c &= L_h^\text{con} h_w + \curlP^\text{con}_h(\curlN^\text{con}(X_c + h(X_c)),
\end{align*}
where $X_c = (U_c,W_c)^T$ and
\begin{align*}
    \curlN^\text{SH}(X_c + h(X_c)) &= (0,0,0,(U_c + h_u(X_c))_0(W_c + h_w(X_c))_0 - (1+\gamma)(U_c+h_u(X_c))_0^3)^T, \\
    \curlN^\text{con}(X_c + h(X_c)) &= (0,2c\gamma (U_c + h_u(X_c))_0 (U_c + h_u(X_c))_1)^T.
\end{align*}
Now, recalling that $c = \varepsilon c_0$ and noting that $h_u, h_w$ are at least quadratic in $X_c$ we find $\curlP^\text{SH}_h(\curlN^\text{SH}(X_c + h(X_c))) = \curlO(\varepsilon^3)$ and $\curlP^\text{con}_h(\curlN^\text{con}(X_c + h(X_c)) = \curlO(\varepsilon^4)$.
Finally, since the hyperbolic eigenvalues have an imaginary part which is uniformly bounded away from zero for all $\varepsilon > 0$, i.e.\ $L_h$ has a bounded inverse we find $h_u = \curlO(\varepsilon^3)$ and $h_w = \curlO(\varepsilon^4)$ as claimed.

To derive the reduced equations on the center manifold, we recall that the reduced dynamic is determined by
\begin{align}
	\partial_\xi \begin{pmatrix}
		U_c \\ W_c
	\end{pmatrix} = L_c \begin{pmatrix}
		U_c \\ W_c
	\end{pmatrix} + \curlP_c(\curlN(U_c + h_u(U_c, W_c), W_c + h_w(U_c, W_c))),
	\label{eq:reducedDynamics}
\end{align}
where $\curlP_c$ is the projection onto the central eigenspace, see \eqref{eq:centralProjection}, and the nonlinearity $\curlN$ is given in \eqref{eq:fullSystem}.
To obtain an equation for $A, B$ we now project onto the eigenspaces spanned by $\varphi_1^{\pm}$ respectively, using the projection
\begin{align*}
	P_{\varphi_1^{\pm}}((U_1,W_1)^T) = \dfrac{\scalarprod{\psi_1^\pm}{(U_1,W_1)^T}}{\scalarprod{\psi_1^{\pm}}{\varphi_1^{\pm}}} \varphi_1^\pm.
\end{align*}
Here, $\scalarprod{\cdot}{\cdot}$ is the euclidean scalar product on $\C^6 \times \C^6$ and $\psi_1^{\pm}$ is the eigenvector of the adjoint matrix $L_1^{\ast}$ corresponding to the eigenvector $\overline{\lambda_{1,-}^{\pm}}$.
We normalize $\psi_1^{\pm} = ((\psi_1^{\pm})_0,(\psi_1^{\pm})_1,(\psi_1^{\pm})_2,(\psi_1^{\pm})_3,0,0)^T$ such that $(\psi_1^{\pm})_3 = 1$, which yields that
\begin{align*}
	&P_{\varphi_1^{\pm}}(\curlN_1(U_c + h_u(U_c,W_c),W_c + h_w(U_c,W_c))) \\
	&\qquad = \dfrac{1}{\scalarprod{\psi_1^{\pm}}{\varphi_1^{\pm}}} \left(\varepsilon^3 (A+B) W_{c0} - 3 (1+\gamma)(A+B)\snorm{A+B}^2 + \curlO(\varepsilon^5)\right) \varphi_1^{\pm},
\end{align*}
where we also used again the first component of $\varphi_1^{\pm}$ is equal to one.
Therefore, applying $P_{\varphi_1^{\pm}}$ to \eqref{eq:reducedDynamics} we find
\begin{align*}
	\varepsilon^2 \partial_{\tilde{\xi}} A &= \varepsilon \lambda_1^+ A + \dfrac{1}{\scalarprod{\psi_1^{+}}{\varphi_1^{+}}} \left(\varepsilon^3 (A+B) W_{c0} - 3 (1+\gamma)(A+B)\snorm{A+B}^2 + \curlO(\varepsilon^5)\right), \\
	\varepsilon^2 \partial_{\tilde{\xi}} B &= \varepsilon \lambda_1^- B + \dfrac{1}{\scalarprod{\psi_1^{-}}{\varphi_1^{-}}} \left(\varepsilon^3 (A+B) W_{c0} - 3 (1+\gamma)(A+B)\snorm{A+B}^2 + \curlO(\varepsilon^5)\right),
\end{align*}
where $\tilde{\xi} = \varepsilon\xi$.
Next, we calculate the normalizing constant $\scalarprod{\psi_1^\pm}{\varphi_1^\pm}^{-1}$.
For that we use that $\scalarprod{\psi_1^\pm}{\varphi_1^\pm} = -p_1'(\lambda_{1,-}^\pm)$, the derivative of the negative characteristic polynomial of $L_1^\text{SH}$, see \cite[Appendix C]{eckmannWayne91} and recall that $\psi_1^\pm$ is normalized such that the last non-zero component is one.
Using the characteristic polynomial given in Appendix \ref{app:eigenvalues} and that $\lambda_{1,-}^\pm = \varepsilon (-c_0\pm \Delta)/8 + \curlO(\varepsilon)$ with $\Delta = \sqrt{c_0^2-16\alpha_0}$, we find
\begin{align*}
	\scalarprod{\psi_1^\pm}{\varphi_1^\pm} = \mp \varepsilon \Delta + \curlO(\varepsilon^2),
\end{align*}
and thus,
\begin{align*}
	\scalarprod{\psi_1^{\pm}}{\varphi_1^{\pm}}^{-1} = \dfrac{\mp 1}{\varepsilon\Delta}(1+ \curlO(\varepsilon)).
\end{align*}
Then, the reduced equations for $A, B$ are given by
\begin{subequations}
	\begin{align}
		\partial_{\tilde{\xi}} A &= \dfrac{-c_0+\Delta}{8} A - \dfrac{1}{\Delta} \left((A+B)W_{c0} - 3(1+\gamma) (A+B) \snorm{A+B}^2\right) + \curlO(\varepsilon^2), \label{eq:CMdynA}\\
		\partial_{\tilde{\xi}} B &= \dfrac{-c_0-\Delta}{8} B + \dfrac{1}{\Delta} \left((A+B)W_{c0} - 3(1+\gamma) (A+B) \snorm{A+B}^2\right) + \curlO(\varepsilon^2) \label{eq:CMdynB}
	\end{align}
\end{subequations}

To derive an equation for $W_{c0},W_{c1}$ we use that the two central eigenvalues originating from the conservation law are the two eigenvalues of $L_0^{\text{con}}$.
Next, we recall from \eqref{eq:spatDynSys} that
\begin{align*}
	\partial_\xi \begin{pmatrix}
		W_{00} \\ W_{01}
	\end{pmatrix} = L_0^{\text{con}} \begin{pmatrix}
		W_{00} \\ W_{01}
	\end{pmatrix} + \begin{pmatrix}
		0 \\ 2 \varepsilon c_0 \gamma \sum_{p+q=0} U_{p0} U_{q1}.
	\end{pmatrix}
\end{align*}
Furthermore, we recall that $U_{nj} = \partial_\xi^j \curlU_n$ and $\curlU_n = \bar{\curlU}_{-n}$, which yields
\begin{align*}
	\sum_{p + q = 0} U_{p0}U_{q0} = 2 \curlU_0 \partial_\xi \curlU + \sum_{p \in \N} (\curlU_p \overline{\partial_\xi \curlU_p} + \bar{\curlU}_p \partial_\xi \curlU_p) = \partial_\xi \sum_{p \in \N_0} \snorm{\curlU_{p0}}^2,
\end{align*}
with $\N_0 = \N \cup \{0\}$.
This yields that
\begin{align*}
	\partial_\xi^2 W_{00} = -\varepsilon c_0 \partial_\xi W_{00} + 2 \varepsilon c_0 \gamma \partial_\xi \sum_{p \in \N_0} \snorm{U_{p0}}^2,
\end{align*}
which we integrate once with respect to $\xi$ to obtain
\begin{align*}
	\partial_\xi W_{00} = -\varepsilon c_0 W_{00} + 2 \varepsilon c_0f \gamma \sum_{p \in \N_0} \snorm{U_{p0}}^2.
\end{align*}
Note that using this integration we have eliminated the central eigenvalue at zero.
Therefore, in the above coordinates the equation on the center manifold for $W_{c0}$ then reads as
\begin{align*}
	\varepsilon^3 \partial_{\tilde{\xi}} W_{c0} = - \varepsilon^3 c_0 W_{c0} + 2 \varepsilon^3 c_0 \gamma \snorm{A+B}^2 + \curlO(\varepsilon^6).
\end{align*}

Now, following \cite[Section 3]{eckmannWayne91}, we substitute $A,B$ in \eqref{eq:CMdynA}--\eqref{eq:CMdynB} by
\begin{align*}
	A = \dfrac{1}{2}\left(\hat{A} + \dfrac{c_0}{\Delta} \hat{A} + \dfrac{8}{\Delta} \tilde{B}\right), \quad B = \dfrac{1}{2}\left(\hat{A} - \dfrac{c_0}{\Delta} \hat{A} - \dfrac{8}{\Delta} \tilde{B}\right),
\end{align*}
and obtain, by additionally defining $\hat{B} := \partial_{\tilde{\xi}} \hat{A} = \tilde{B} + \curlO(\varepsilon^2)$ and $\hat{W}_0 := W_{c0}$, that
\begin{subequations}
	\begin{align}
	\partial_{\tilde{\xi}}  \hat{A} &= \hat{B}, \label{eq:reducedSH1}\\
	\partial_{\tilde{\xi}} \hat{B} &= \dfrac{1}{4}\left(-\alpha_0 \hat{A} - c_0 \hat{B} - \hat{A} \hat{W}_0 + 3(1+\gamma) \hat{A} \snorm{\hat{A}}^2\right) + \curlO(\varepsilon^2), \label{eq:reducedSH2}\\
	\partial_{\tilde{\xi}} \hat{W}_0 &= -c_0 \hat{W}_0 + 2 c_0 \gamma \snorm{\hat{A}}^2 + \curlO(\varepsilon^3) \label{eq:reducedCon}.
	\end{align}
\end{subequations}
This system then determines the dynamic on the center manifold.

\begin{rem}
	Note that although we have 6 central eigenvalues, the reduced system \eqref{eq:reducedSH1}--\eqref{eq:reducedCon} is 3-dimensional.
	As explained in the derivation, by integrating the equation for $W_{00}$ once with respect to $\xi$ we removed the zero eigenvalue by using the conservation property.
	Furthermore, since the solutions of \eqref{eq:SHe}--\eqref{eq:ConE} are supposed to be real, we have $U_j = \bar{U}_{-j}$ for $j \in \Z$. This means that the remaining two central eigenvalues $\lambda_{-1,+}^{\pm}$ yield the complex conjugated equations of \eqref{eq:reducedSH1}--\eqref{eq:reducedSH2}.
\end{rem}

\subsection{Construction of a heteroclinic orbit}

We now show that the ODE system \eqref{eq:reducedSH1}--\eqref{eq:reducedCon} has a circle of nontrivial fixed points.
We also show that for $(\varepsilon,\gamma) = (0,0)$ the system exhibits heteroclinic orbits connecting the nontrivial fixed points to the origin.
Finally, we prove that these orbits are persistent for $(\gamma,\varepsilon)$ in a sufficiently small neighborhood of zero, which gives the main result, Theorem \ref{thm:modFronts}.

First, we study the fixed points of \eqref{eq:reducedSH1}--\eqref{eq:reducedCon}, provided that $c_0^2 > 16\alpha_0$.
The system has the trivial fixed point $(0,0,0)$.
Furthermore, for $\varepsilon = 0$ and $\gamma > -3$ we find a circle of nontrivial fixed points at
\begin{align*}
	\left(\sqrt{\dfrac{1}{3+\gamma} \alpha_0} e^{i\phi}, 0, \dfrac{2\gamma}{3+\gamma}\alpha_0\right),
\end{align*}
where $\phi \in [0,2\pi)$. 
Using the implicit function theorem similar to the proof of Lemma \ref{lem:periodicSols} these fixed points also persist for $\varepsilon > 0$ sufficiently small.
Note that these fixed points correspond to the periodic solutions in Lemma \ref{lem:periodicSols} since inserting the fixed points into the transformation for $v$ gives
\begin{align*}
	v = w - \gamma u^2 = \varepsilon^2 \hat{W}_0 - 2\gamma \varepsilon^2 \snorm{\hat{A}}^2 + \curlO(\varepsilon^3) = \curlO(\varepsilon^3).
\end{align*}

Next, we set $(\gamma,\varepsilon) = (0,0)$.
Then, the origin is a hyperbolic stable fixed point since the linearization has real, negative spectrum.
Additionally, splitting the system in real and imaginary parts -- recall that $\hat{W}_0$ is real -- we find that the circle of nontrivial fixed points forms a normally hyperbolic invariant set with one positive, one zero and three negative eigenvalues.
The existence of a heteroclinic orbit connecting the circle of nontrivial fixed points with the origin follows from the fact that $(\hat{A},\hat{B},0)$ is an invariant set and thus, the system reduces to the system studied in \cite{eckmannWayne91}.
Therefore, the heterclinic orbit can be found by phase plane analysis.

We now show that these orbits persist for $(\gamma,\varepsilon)$ in a small neighborhood of zero.
The idea is to show that for $(\gamma, \varepsilon) = (0,0)$ the unstable manifold of the nontrivial fixed point and the stable manifold of the origin intersect transverally.
Therefore, the intersection is stable with respect to small perturbations and therefore, the intersection persists for $(\gamma,\varepsilon)$ in a small neighborhood of zero.
Therefore, we define
\begin{align*}
	H(\hat{A},\hat{B}) := 2\snorm{\hat{B}}^2 + \dfrac{\alpha_0}{2} \snorm{\hat{A}}^2 - \dfrac{3}{4} \snorm{\hat{A}}^4,
\end{align*}
which is a Lyapunov function for the system \eqref{eq:reducedSH1}--\eqref{eq:reducedSH2} for $(\gamma,\varepsilon) = (0,0)$ and $\hat{W}_0 = 0$.
Differentiation with respect to $\xi$ of $H$ along a solution yields $\partial_\xi H = -c_0 \snorm{\hat{B}}^2 \leq 0$.
Furthermore, we remark the system has no fixed points with $\snorm{\hat{A}}^2 < \alpha_0/3$ except for the origin.
Thus, the stable manifold of the origin, denoted by $\curlM_s(0,0,0)$, contains the set
\begin{align*}
	S_0 = \left\{(\hat{A},\hat{B},\hat{W}_0) \in \C^3 : \hat{W}_0 = 0 \text{ and } H(s\hat{A},s\hat{B}) < H\left(\sqrt{\dfrac{\alpha_0}{3}},0\right) \text{ for all } s \in [0,1]\right\}.
\end{align*}
We now show that the unstable manifold of the nontrivial fixed point $\curlM_u(A_f,B_f,W_f)$ intersects $\curlM_s(0,0,0)$ transversally.
Therefore, we note that $\curlM_u(A_f,B_f,W_f)$ lies in the set $(\hat{A},\hat{B},0) \in \C^3$ since this is invariant and $\curlM_u(A_f,B_f,W_f)$ is tangential to the unstable eigenspace of the nontrivial fixed points which has no contribution in $\hat{W}_0$-direction for $(\gamma,\varepsilon) = (0,0)$.
Moreover, the unstable eigenspace of the nontrivial fixed point intersects $S_0$ transversally in $\C^2 \times \{0\}$.
Using the stable manifold theorem (see e.g.\ \cite{perko01}) and the fact that the stable space of the linearization about the origin is three dimensional, we conclude that $\curlM_s(0,0,0)$ is a smooth, three dimensional manifold in $\C^3$.
Therefore, for every point $(\hat{A},\hat{B},0) \in S_0$ there exists a small neighborhood in $\C^3$ which is contained in $\curlM_s(0,0,0)$.
Hence, $\curlM_s(0,0,0)$ and $\curlM_u(A_f,B_f,W_f)$ intersect transversally and we obtain the persistence of the heteroclinic connections for small perturbations.
In particular, this includes $(\gamma, \varepsilon)$ in a small neighborhood of zero.

Now using the transformation $W = V + \gamma U^2$, we obtain the existence of modulating traveling fronts for \eqref{eq:SHe}--\eqref{eq:ConE}.
In particular, this also yields an approximate representation of these solutions.
Thus, we proved the following theorem.

\begin{theorem}\label{thm:modFronts}
	Let $c_0^2 > 16\alpha_0 > 0$.
	Then there exists $\varepsilon_0 > 0$, $\gamma_0 > 0$ such that for all $\varepsilon < \varepsilon_0$ and $\snorm{\gamma} < \gamma_0$ the system~\eqref{eq:SHe}--\eqref{eq:ConE} has a modulating travelling front solution $(u_f, v_f)$.
	Furthermore, it holds that
	\begin{align*}
		u_f(t,x) &= \varepsilon 2 \snorm{\hat{A}(y)} \cos(x + x_0) + \mathcal{O}(\varepsilon^2), \\
		v_f(t,x) &= \varepsilon^2\left[\hat{W}_0(y) - 2\gamma \snorm{\hat{A}(y)}^2 - 2 \gamma \snorm{\hat{A}(y)}^2 \cos(2 (x + x_0))\right] + \mathcal{O}(\varepsilon^3),
	\end{align*}
	in $L^\infty$ where $y = \varepsilon x - \varepsilon^2 c_0 t$, $x_0 \in [0,2\pi)$ and $(\hat{A},\hat{W}_0)$ are a heteroclinic solution of~\eqref{eq:reducedSH1}--\eqref{eq:reducedCon}.
\end{theorem}

\begin{figure}
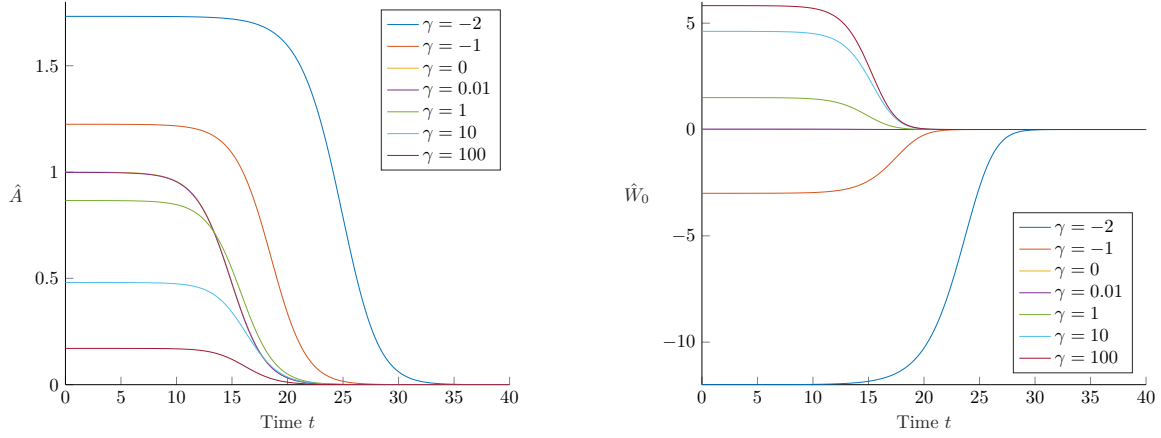

	\centering
	\include{HeteroclinicOrbits}
	\caption{Numerical solutions of reduced system \eqref{eq:reducedSH1}--\eqref{eq:reducedCon} for $c_0 = 7$, $\alpha_0 = 3$.}
	\label{fig:heteroclinicOrbits}
\end{figure}

\begin{rem}
	Recall that we set $k_c = 1$ in the calculations above.
	However, the stationary, periodic solutions exist also for $k_c$ close to 1.
	Therefore, we expect that a similar results holds in the setting of Lemma \ref{lem:periodicSols}.
\end{rem}

We stress that the above result is only valid for $\gamma$ close to zero.
However, we expect that this is a restriction only imposed due to the use of perturbation arguments.
A first intuition comes from the fact that the circle of nontrivial fixed points exists for all $\gamma > -3$.
This is backed by Figure \ref{fig:heteroclinicOrbits} which shows a numerical simulation of heteroclinic orbits for different values of $\gamma$.
For this we approximated an initial point on the unstable manifold of the nontrivial fixed point using the unstable eigenvector of the linearization about this fixed point.
Then, using this approximation, we solved the forward dynamics given by the reduced system \eqref{eq:reducedSH1}--\eqref{eq:reducedCon}.
It turns out that the spectral stability of the nontrivial fixed point does not change, i.e.\ the circle of nontrivial fixed points is normally hyperbolic with exactly one unstable direction.
Furthermore, the numerical calculation gives heteroclinic orbits of the system which hints the existence of such solutions.

Finally, the numerical calculation hints that for large $\gamma$ the solution is close to a heteroclinic orbit of a limiting system which can be constructed by introducing the rescaled variables
\begin{align*}
	\begin{pmatrix}
	\tilde{A} \\ \tilde{B}
	\end{pmatrix} := \sqrt{\dfrac{3+\gamma}{\alpha_0}} \begin{pmatrix}
	\hat{A} \\ \hat{B}
	\end{pmatrix} \text{ and } \tilde{W}_0 := \dfrac{3+\gamma}{2\alpha_0 \gamma} \hat{W}_0.
\end{align*}
Inserting this into \eqref{eq:reducedSH1}--\eqref{eq:reducedCon} and formally passing to $\gamma \rightarrow \infty$ we find the formal asymptotic system
\begin{align*}
	\partial_\xi \tilde{A} &= \tilde{B} \\
	\partial_\xi \tilde{B} &= \dfrac{1}{4}\left(-c_0 \tilde{B} - \alpha_0 \tilde{A} - 2\alpha_0 \tilde{A} \tilde{W}_0 + 3\alpha_0 \tilde{A}\snorm{\tilde{A}} \right) \\
	\partial_\xi \tilde{W}_0 &= -c_0 \tilde{W_0} + c_0 \snorm{\tilde{A}}^2.
\end{align*}
As expected this system has a trivial fixed point at $(0,0,0)$ and a circle of nontrivial fixed points at $(e^{i\phi},0,1)$ for $\phi \in [0,2\pi)$.
However, we note that this asymptotic model has limited use for the description of the full dynamics \eqref{eq:SHe}--\eqref{eq:ConE} since the size of the center manifold vanishes for $\gamma \rightarrow \infty$.

\section{Discussion} \label{sec:discussion}

\begin{figure}
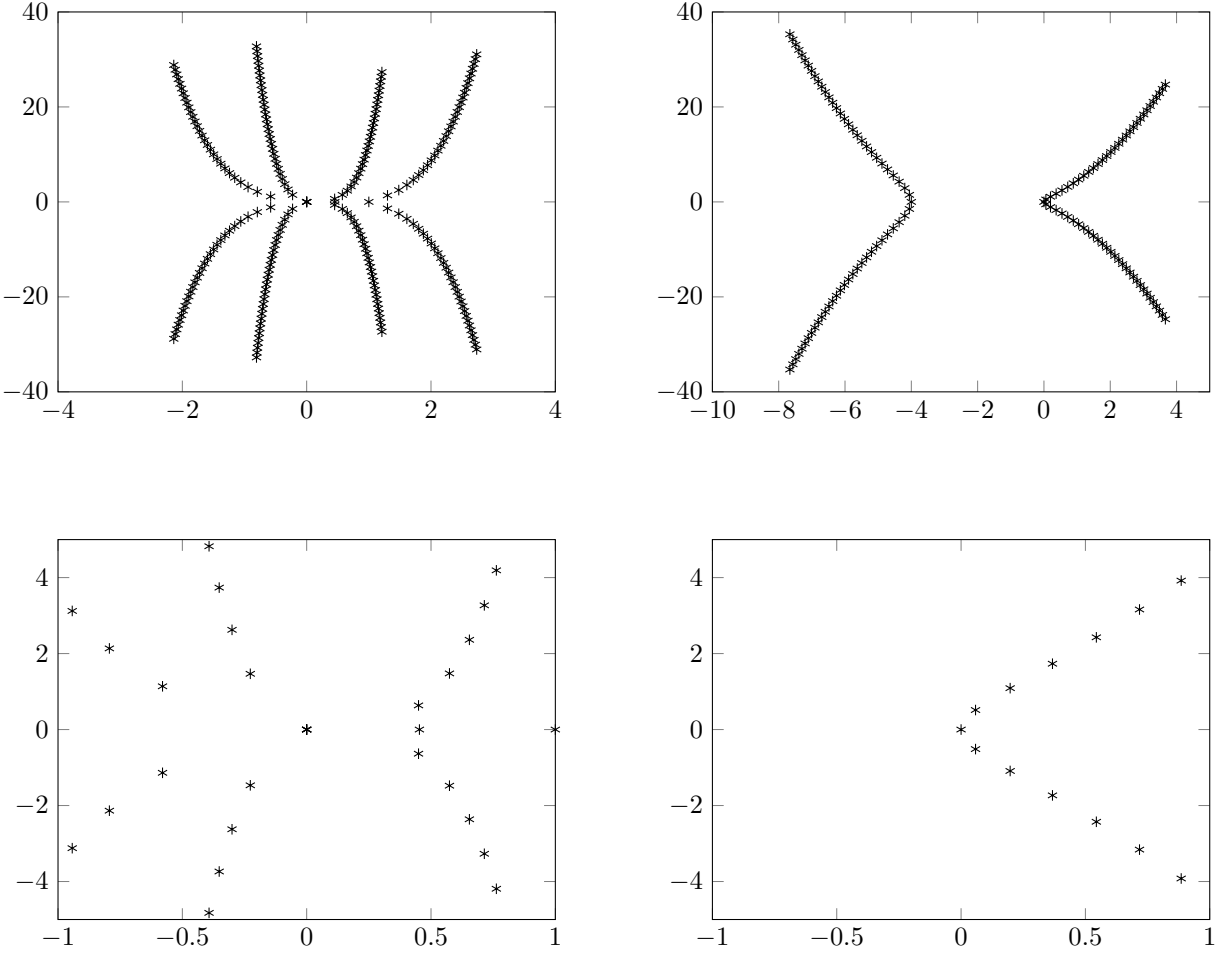

	\centering
	\include{specDispersion}
	\caption{Numerical calculation of the eigenvalues of $L_n^\text{SH}$ and $L_n^\text{con}$ corresponding to \eqref{eq:extendedSH} and \eqref{eq:extendedCon}, respectively, for $\alpha_0 = 1$, $c_u,c_v = 1$, $c = 3$ and $\varepsilon = 0$. 
		Top left and right show the eigenvalues of $L_n^\text{SH}$ and $L_n^\text{con}$ for $-30 \leq n \leq 30$, respectively. 
		Bottom left and right are rescaled versions of the top left and top right plots, respectively, centered around the origin.}
	\label{fig:spectrumExtended}
\end{figure}

We now discuss some possible extensions and open questions.

\paragraph{Adding dispersion.}
In this paper we only  consider the case of vanishing phase velocities $\beta = 0$, i.e.\ that the periodic pattern is stationary.
However, it is also interesting to study systems where this is not the case.
One example of such systems is
\begin{subequations}
	\begin{align}
		\partial_t u &= -(1+\partial_x^2)^2 u + \varepsilon^2 \alpha_0 u + c_u \partial_x^3 u + uv + u \partial_x u - u^3 \label{eq:extendedSH}\\
		\partial_t v &= \partial_x^2 v + c_v \partial_x v + \gamma_1 \partial_x^2 (u^2) + \gamma_2 \partial_x (u^2) \label{eq:extendedCon},
	\end{align}
\end{subequations}
which is a toy model corresponding to a thin-film flow on an inclined, heated surface.
Since the spectrum of the linearisation around $(u,v) = (0,0)$ has non-zero imaginary part, we can construct non-stationary periodic solutions of the form $(u,v)(t,x) = (U,V)(x-\beta t)$ using center manifold theory similar to Lemma \ref{lem:periodicSols}.
Here, $\beta = c_u + \curlO(\varepsilon)$ denotes the phase velocity.
A similar situation also occurs in the Taylor-Couette problem in the case of counter-rotating cylinders, see \cite{ioossMielke91,haragusSchneider99}.

It turns out that due to the dispersion in \eqref{eq:extendedSH}--\eqref{eq:extendedCon} the spectral situation in the construction of modulating front simplifies.
Making a modulating front ansatz $(u,v)(t,x) = (U,V)(x-ct,x-\beta t)$ and rewriting the system as a first order system in $\xi = x-ct$ similar to Section \ref{subsec:spatialformulation} we find that the there is an $\curlO(1)$-spectral gap in the spectrum (see Figure \ref{fig:spectrumExtended}).
Therefore, standard center manifold theory is applicable.
A closer inspection of the central spectrum shows that depending on the front velocity $c$ different situations appear.
For $c \neq 3c_u + o(1)$ we obtain a 3-dimensional center manifold while for $c$ close to $3c_u$ (with respect to $\varepsilon$) the center manifold is 5-dimensional.
Note that $3c_u$ is the group velocity of the periodic solution corresponding to the critical Fourier mode $k_c = \pm 1$.
Thus, the first case corresponds the case that the front either moves faster or slower than the group velocity while the second case corresponds to the front velocity being close to the group velocity.
However, although the center manifold can be established with standard theory in this case, the difficulty for the existence of modulating fronts lies in the construction of heteroclinic orbits for the reduced model since the presence of dispersion leads to a reduced equation with complex coefficients.

\paragraph{Nonlinear stability.}
Another largely open question is the stability of the solutions presented here and we briefly discuss existing results in this direction here.
We start by discussing stability of the periodic solutions established in Lemma \ref{lem:periodicSols}.
Spectral stability has already been discussed in \cite{matthewsCox00, sukhtayev16ArXiv} for a similar model using Bloch-Floquet theory and we expect that similar calculations apply in our case.
Furthermore, in recent years there were many results concerning the diffusive stability of periodic solutions in systems of conservation laws and we refer to \cite{johnsonZumbrun10,johnsonZumbrun11,barkerJohnsonNobleRodriguesZumbrun13} for details.
Therefore, we conjecture that stability of the periodic solutions from Lemma \ref{lem:periodicSols} can be established.
However, a rigorous proof of this stability is an open problem.

Now we turn to the modulating fronts.
The main problem for establishing stability of these solutions is the fact that they connect an unstable state, the origin, to a (possibly) stable state, the periodic solution.
In the case of the Swift-Hohenberg equation \cite{eckmannSchneider02} and the Taylor-Couette problem \cite{eckmannSchneider00} nonlinear stability with respect to exponentially localized perturbations has already been established.
The key idea in both papers is to introduce exponentially weighted variables to stabilize the origin in the Bloch-transformed problem and then applying renormalization group theory to close the argument.
Whether these ideas can also be used to establish stability in the case of an additional conservation law is open, however.
Especially, following the discussion in \cite{barkerJohnsonNobleRodriguesZumbrun13}, it is unclear whether renormalization groups can still be used to close the argument or if other tools are necessary.

\section*{Acknowledgments}
The author thanks Bj{\"o}rn de Rijk and Guido Schneider for valuable discussions.
This work was supported by the German Research Foundation (DFG) within the Cluster of Excellence in Simulation Technology (EXC 310/2) at the University of Stuttgart.

\appendix
\section{Calculation of eigenvalues}\label{app:eigenvalues}

We now give the calculations, which lead to the expansions \eqref{eq:eigenvalueSH1}--\eqref{eq:eigenvalueCon} of the eigenvalues of $L_n$, $n \in \Z$.
Therefore, recall that $L_n$ is a block-diagonal matrix with
\begin{align*}
	L_n^\text{SH} = \begin{pmatrix}
		0 & 1 & 0 & 0 \\
		0 & 0 & 1 & 0 \\
		0 & 0 & 0 & 1 \\
		-(1+n^2)^2 + \varepsilon^2 \alpha_0 & -4in(1-n^2)+\varepsilon c_0 & 6n^2-2 & -4in
	\end{pmatrix}
\end{align*}
and
\begin{align*}
	L_n^\text{con} = \begin{pmatrix}
		0 & 1 \\
		n^2 & 2in-\varepsilon c_0
	\end{pmatrix}.
\end{align*}

\subsection{Eigenvalues of $L_n^\text{SH}$}

We first calculate the approximation of the eigenvalues of $L_n^\text{SH}$, see \eqref{eq:eigenvalueSH1}--\eqref{eq:eigenvalueSH2}.
According to \cite[Appendix C]{eckmannWayne91} the characteristic polynomial is given by
\begin{align*}
	p_n^\text{SH}(\lambda,\varepsilon) &= -\lambda^4 -4in\lambda^3 + \lambda^2(6n^2-2) + \lambda(-4in(1-n^2)+\varepsilon c_0) -(1+n^2)^2 + \varepsilon^2 \alpha_0 \\
	&= -(\lambda+i(n+1))^2 (\lambda+i(n-1))^2 + \varepsilon c_0 \lambda + \varepsilon^2 \alpha_0.
\end{align*}
Hence for $\varepsilon = 0$, $L_n^\text{SH}$ has two double eigenvalues $\lambda_{n,\pm} = -i(n\pm 1)$.

We start by considering the case $\lambda_{1,-} = 0$ and define
\begin{align*}
	\tilde{p}^\text{SH}_{1,-}(\delta,\varepsilon) := \varepsilon^{-2} p_1^\text{SH}(\lambda_{1,-}+\varepsilon\delta,\varepsilon) = -\delta^2 (\varepsilon\delta + 2i)^2 + c_0 \delta + \alpha_0.
\end{align*}
Setting $\varepsilon = 0$ and solving for $\delta$ yields
\begin{align*}
	\delta_{\pm} = \dfrac{-c_0 \pm \Delta}{8}, \quad \Delta = \sqrt{c_0^2-16\alpha_0}.
\end{align*}
Since $\tilde{p}_{1,-}^\text{SH}$ is continuously differentiable,
\begin{align*}
	\tilde{p}_{1,-}^\text{SH}(\delta_{\pm},0) = 0 \text{ and } \partial_\delta \tilde{p}_{1,-}^\text{SH}(\delta_{\pm},0) \neq 0
\end{align*}
we can apply the implicit function theorem, which yields that for $\varepsilon > 0$ small there exists a root of $p_1^\text{SH}$ of the form
\begin{align*}
	\lambda_{1,-}^{\pm} = \varepsilon \dfrac{-c_0\pm \Delta}{8} + \curlO(\varepsilon^2).
\end{align*}
The same calculation in the case $\lambda_{-1,+} = 0$ yields
\begin{align*}
	\lambda_{-1,+}^{\pm} = \varepsilon \dfrac{-c_0 \pm \Delta}{8} + \curlO(\varepsilon^2).
\end{align*}

Next, we consider the case that $\lambda_{n,\pm} \neq 0$.
Similar to the case above, we define
\begin{align*}
	\tilde{p}_{n,\pm}^\text{SH}(\delta,\tilde{\varepsilon}) = \tilde{\varepsilon}^{-2} p_{n,\pm}^\text{SH}(-i(n\pm1 ) + \tilde{\varepsilon}\delta, \tilde{\varepsilon}^2) = -\delta^2(\mp 2i + \tilde{\varepsilon}\delta)^2 + c_0 (-i(n\pm 1) + \tilde{\varepsilon} \delta) + \tilde{\varepsilon}^2 \alpha_0,
\end{align*}
where we set $\varepsilon = \tilde{\varepsilon}^2$ so that $\tilde{p}_{n,\pm}$ is continuously differentiable in both arguments.
Setting $\tilde{\varepsilon} = 0$ and solving for $\delta$ then gives
\begin{align*}
	\delta_{\pm}^+ = i^{3/2} \dfrac{\sqrt{(n\pm 1)c_0}}{2} \text{ and } \delta_{\pm}^- = - \delta_{\pm}^+.
\end{align*}
Again the implicit function theorem is applicable and yields that for $\varepsilon > 0$ small the roots of $p_{n,\pm}^\text{SH}$ are of the form
\begin{align*}
	\lambda_{n,\pm}^+ &= -i(n \pm 1) + \varepsilon^{1/2} i^{3/2} \dfrac{\sqrt{(n\pm 1) c_0}}{2} + \curlO(\varepsilon), \\
	\lambda_{n,\pm}^- &= -i(n \pm 1) - \varepsilon^{1/2} i^{3/2} \dfrac{\sqrt{(n\pm 1) c_0}}{2} + \curlO(\varepsilon)
\end{align*}
Hence, we obtained \eqref{eq:eigenvalueSH1}--\eqref{eq:eigenvalueSH2}.

\subsection{Eigenvalues of $L_n^\text{con}$}
Let $n \neq 0$.
We now give the calculations leading up to \eqref{eq:eigenvalueCon}. In this case, the characteristic polynomial is given by
\begin{align*}
	p^\text{con}_n(\nu,\varepsilon) = (\nu-in)^2 + \varepsilon c_0 \nu.
\end{align*}
For $\varepsilon = 0$, we find a double root at $\nu_n = in$.
For $\varepsilon > 0$, we define
\begin{align*}
	\tilde{p}^\text{con}_n(\delta, \tilde{\varepsilon}) = \tilde{\varepsilon}^{-2} p_n^\text{con}(in+\tilde{\varepsilon}\delta, \tilde{\varepsilon}^2) = \delta^2 + \tilde{\varepsilon} c_0 \delta + c_0 in,
\end{align*}
which is continuously differentiable in $(\delta,\tilde{\varepsilon})$.
Next, solving $\tilde{p}^\text{con}_n(\delta,0)=0$ for $\delta$ gives two solutions $\delta_{\pm} = \pm i^{3/2} \sqrt{nc_0}$.
The implicit function theorem then yields that for $\varepsilon > 0$ small the eigenvalues of $L_n^\text{con}$ for $n \neq 0$ are approximated by
\begin{align*}
	\nu_{n,\pm} = in \pm \varepsilon^{1/2} i^{3/2} \sqrt{nc_0} + \curlO(\varepsilon),
\end{align*}
which is \eqref{eq:eigenvalueCon}.

\phantomsection
\addcontentsline{toc}{section}{References}
\bibliography{BibDeskLibrary.bib}

\newcommand{\etalchar}[1]{$^{#1}$}
\begin{thebibliography}{BJN{\etalchar{+}}13}

\bibitem[BJN{\etalchar{+}}13]{barkerJohnsonNobleRodriguesZumbrun13}
B.~Barker, M.~A. Johnson, P.~Noble, L.~M. Rodrigues, and K.~Zumbrun.
\newblock Nonlinear modulational stability of periodic traveling-wave solutions
  of the generalized {K}uramoto-{S}ivashinsky equation.
\newblock {\em Physica D}, 258:11--46, 2013.

\bibitem[CE86]{colletEckmann86}
P.~Collet and J.-P. Eckmann.
\newblock The existence of dendritic fronts.
\newblock {\em Communications in Mathematical Physics}, 107(1):39--92, 1986.

\bibitem[CH93]{crossHohenberg93}
M.~C. Cross and P.~C. Hohenberg.
\newblock Pattern formation outside of equilibrium.
\newblock {\em Reviews of modern physics}, 65(3):851--1112, 1993.

\bibitem[CM03]{coxMatthews03}
S.~M. Cox and P.~C. Matthews.
\newblock Instability and localisation of patterns due to a conserved quantity.
\newblock {\em Physica D}, 175(3):196--219, 2003.

\bibitem[DKSZ16]{duellKashaniSchneiderZimmermann16}
W.-P. D{\"u}ll, K.~S. Kashani, G.~Schneider, and D.~Zimmermann.
\newblock Attractivity of the {G}inzburg--{L}andau mode distribution for a
  pattern forming system with marginally stable long modes.
\newblock {\em Journal of Differential Equations}, 261(1):319--339, 2016.

\bibitem[ES00]{eckmannSchneider00}
J.-P. Eckmann and G.~Schneider.
\newblock Nonlinear stability of bifurcating front solutions for the
  {T}aylor‐{C}ouette problem.
\newblock {\em Zeitschrift f{\"u}r angewandte Mathematik und Mechanik ZAMM},
  80(11--12):745--753, 2000.

\bibitem[ES02]{eckmannSchneider02}
J.-P. Eckmann and G.~Schneider.
\newblock Non-linear stability of modulated fronts for the {S}wift--{H}ohenberg
  equation.
\newblock {\em Communications in Mathematical Physics}, 225(2):361--397, 2002.

\bibitem[EW91]{eckmannWayne91}
J.-P. Eckmann and C.~E. Wayne.
\newblock Propagating fronts and the center manifold theorem.
\newblock {\em Communications in Mathematical Physics}, 136(1):285--307, 1991.

\bibitem[FH15]{fayeHolzer15}
G.~Faye and M.~Holzer.
\newblock Modulated traveling fronts for a nonlocal {F}isher-{KPP} equation: A
  dynamical systems approach.
\newblock {\em Journal of Differential Equations}, 258(7):2257--2289, 2015.

\bibitem[HCS99]{haragusSchneider99}
M.~H{\u{a}}r{\u{a}}gu{\c{s}}-Courcelle and G.~Schneider.
\newblock Bifurcating fronts for the {Taylor-Couette} problem in infinite
  cylinders.
\newblock {\em Zeitschrift f{\"u}r angewandte Mathematik und Physik ZAMP},
  50(1):120--151, 1999.

\bibitem[HI11]{haragusIooss11}
M.~Haragus and G.~Iooss.
\newblock {\em Local Bifurcations, Center Manifolds, and Normal Forms in
  Infinite-Dimensional Dynamical Systems}.
\newblock Springer-Verlag London, 2011.

\bibitem[HSZ11]{haeckerSchneiderZimmermann11}
T.~H{\"a}cker, G.~Schneider, and D.~Zimmermann.
\newblock Justification of the {G}inzburg--{L}andau approximation in case of
  marginally stable long waves.
\newblock {\em Journal of Nonlinear Science}, 21(1):93--113, 2011.

\bibitem[IM91]{ioossMielke91}
G.~Iooss and A.~Mielke.
\newblock Bifurcating time--periodic solutions of {N}avier--{S}tokes equations
  in infinite cylinders.
\newblock {\em Journal of Nonlinear Science}, 1(1):107--146, 1991.

\bibitem[JZ10]{johnsonZumbrun10}
M.~A. Johnson and K.~Zumbrun.
\newblock Nonlinear stability of periodic traveling wave solutions of systems
  of viscous conservation laws in the generic case.
\newblock {\em Journal of Differential Equations}, 249(5):1213--1240, 2010.

\bibitem[JZ11]{johnsonZumbrun11}
M.~A. Johnson and K.~Zumbrun.
\newblock Nonlinear stability of periodic traveling-wave solutions of viscous
  conservation laws in dimensions one and two.
\newblock {\em SIAM Journal of Applied Dynamical Systems}, 10(1):189--211,
  2011.

\bibitem[Kno16]{knobloch16}
E.~Knobloch.
\newblock Localized structures and front propagation in systems with a
  conservation law.
\newblock {\em IMA Journal of Applied Mathematics}, 81(3):457--487, 2016.

\bibitem[MC00]{matthewsCox00}
P.~C. Matthews and S.~M. Cox.
\newblock Pattern formation with a conservation law.
\newblock {\em Nonlinearity}, 13(4):1293--1320, 2000.

\bibitem[Per01]{perko01}
L.~Perko.
\newblock {\em Differential Equations and Dynamical Systems}.
\newblock Texts in Applied Mathematics. Springer-Verlag New York, 3rd edition,
  2001.

\bibitem[SU17]{schneiderUecker17}
G.~Schneider and H.~Uecker.
\newblock {\em Nonlinear {PDE}s: A Dynamical Systems Approach}, volume 182 of
  {\em Graduate Studies in Mathematics}.
\newblock American Mathematical Soc., 2017.

\bibitem[Suk16]{sukhtayev16ArXiv}
A.~Sukhtayev.
\newblock Diffusive stability of spatially periodic patterns with a
  conservation law.
\newblock {\em arXiv:1610.05395v2}, 2016.

\bibitem[SZ13]{schneiderZimmermann13}
G.~Schneider and D.~Zimmermann.
\newblock Justification of the {G}inzburg--{L}andau approximation for an
  instability as it appears for {M}arangoni convection.
\newblock {\em Mathematical Methods in the Applied Sciences}, 36(9):1003--1013,
  2013.

\bibitem[SZ17]{schneiderZimmermann17}
G.~Schneider and D.~Zimmermann.
\newblock The {T}uring instability in case of an additional conservation law --
  {D}ynamics near the {E}ckhaus boundary and open questions.
\newblock In {\em Patterns of Dynamics, \upshape{pp. 28--43}}. Springer
  International Publishing, 2017.

\bibitem[Zim14]{zimmermann14}
D.~Zimmermann.
\newblock {\em Justification of an Approximation Equation for the
  B{\'e}nard-Marangoni Problem}.
\newblock PhD thesis, Universit{\"a}t Stuttgart, 2014.

\end{thebibliography}
\bibliographystyle{alphainitials}

\end{document}